\documentclass[10pt, a4paper]{amsart}

\usepackage[pdftex]{graphicx}

\usepackage{amsthm}
\newtheorem{theorem}{Theorem}[section]
\newtheorem*{theorem*}{Theorem}
\newtheorem{lemma}[theorem]{Lemma}
\newtheorem*{lemma*}{Lemma}
\newtheorem{proposition}[theorem]{Proposition}
\newtheorem*{proposition*}{Proposition}

\newtheorem*{corollary*}{Corollary}

\newtheorem*{recall*}{Recall}

\newtheorem*{conjecture*}{Conjecture}

\theoremstyle{definition}
\newtheorem{definition}[theorem]{Definition}
\newtheorem*{definition*}{Definition}
\newtheorem{example}[theorem]{Example}
\newtheorem*{example*}{Example}

\newtheorem*{fact*}{Fact}
\newtheorem{remark}[theorem]{Remark}

\newcommand{\LeftEqNo}{\let\veqno\@@leqno}
\numberwithin{equation}{section}

\usepackage{amsmath,amsfonts} % AMSのたくさんの数学用の文字
\usepackage{mathrsfs, txfonts} % 筆記体と花文字

\usepackage{ascmac, fancybox} % 文章を囲う

\newcommand{\R}{\mathbb{R}}

\newcommand{\id}{\mathop{\mathrm{id}}\nolimits}

\newcommand{\Int}{\mathrm{Int}\,}

\newcommand{\Aut}{\mathop{\mathrm{Aut}}\nolimits}

\begin{document}
\title{Computation of the knot symmetric quandle and its application to the plat index of surface-links}

% Only \author and \address are required; other information is
% optional.  Remove any unused author tags.

% author one information
\author{Jumpei Yasuda}
\address{Department of Mathematics, Graduate School of Science, Osaka University,  Toyonaka, Osaka 560-0043, Japan}
% \curraddr{}
\email{u444951d@ecs.osaka-u.ac.jp}
% \thanks{}

% % author two information
% \author{}
% \address{}
% \curraddr{}
% \email{}
% \thanks{}

% \subjclass is required.
\date{}
\dedicatory{}

\begin{abstract}
   A surface-link is a closed surface embedded in the 4-space, possibly disconnected or non-orientable.
   Every surface-link can be presented by the plat closure of a braided surface, which we call a plat form presentation.
   The knot symmetric quandle of a surface-link $F$ is a pair of a quandle and a good involution determined from $F$.
   In this paper, we compute the knot symmetric quandle for surface-links using a plat form presentation.

   As an application, we show that for any integers $g \geq 0$ and $m \geq 2$, there exists infinitely many distinct surface-knots of genus $g$ whose plat indices are $m$.
\end{abstract}

\maketitle

\section{Introduction}\label{Section: Introduction}
A \textit{surface-knot} is a connected, closed surface embedded in $\R^4$, and a \textit{surface-link} is a disjoint union of surface-knots.
% an (unoriented) closed surface embedded in $\R^4$, possibly disconnected and non-orientable, and a \textit{surface-knot} is a connected one.
A \textit{2-knot} is an embedded 2-sphere in $\R^4$.
Two surface-links are said to be \textit{equivalent} to each other if one is ambiently isotopic to the other.

A \textit{quandle} (\cite{Joyce82, Mateev84}) is a non-empty set $X$ with a binary operation satisfying three axioms corresponding to Reidemeister moves.
Quandles are useful for studying oriented classical links and oriented surface-links.
Kamada introduced a \textit{symmetric quandle} (\cite{Kamada2006, Kamada-Oshiro2010}) as a pair of a quandle and its \textit{good involution}.
Symmetric quandles are useful tools for studying unoriented classical links and unoriented surface-links.
As a generalization of the knot group and knot quandle, the \textit{knot symmetric quandle} $X(F)$ of a surface-link $F$ was introduced and studied in \cite{Kamada2006,Kamada2014, Kamada-Oshiro2010}.

In \cite{Rudolph1983}, Rudolph introduced a \textit{braided surface} as an analogy of a braid in surface-knot theory.
A braided surface of degree $m$ with $k$ branch points can be presented by an $r$-tuple of classical $m$-braids, called a \textit{braid system}.
See Definition~\ref{Definition: braid system}.
In \cite{Yasuda21}, it is shown that every surface-link has a \textit{plat form presentation}, which is the plat closure of an \textit{adequate} braided surface.
See \cite{Yasuda21} or Section~\ref{subsection: Braided surfaces and braid systems} for the definition and details.

The aim of this paper is to compute the knot symmetric quandle of a surface-link by using a braid system via a plat form presentation.

\begin{theorem}\label{MainTheoremA}
   Let $F$ be a surface-link, and $S$ be an adequate braided surface $S$ of degree $2m$ providing a plat form for $F$.
   Let $(\beta_1^{-1} \sigma^{\varepsilon_1}_{k_1}\beta_1, \ldots, \beta_r^{-1} \sigma^{\varepsilon_r}_{k_r}\beta_r)$ be a braid system of $S$, where $\sigma_i$ is Artin's generator of $B_{2m}$, $\beta_i \in B_{2m}$, $\varepsilon_i \in \{\pm1\}$, and $1 \leq k_i \leq 2m-1$.
   Then the knot symmetric quandle $X(F)$ has the following presentation:
   \[
      \left\langle x_1, \ldots, x_{2m}~
         \begin{array}{|c}
            \mathrm{Artin}(\beta_j)(x_{k_i}) = \mathrm{Artin}(\beta_j)(x_{k_i+1}) ~(i=1,\ldots,r)\\
            x_{2j-1}=x_{2j}^{-1} ~ (j = 1,\ldots,m)
         \end{array}
       \right\rangle_{\mathrm{sq}},
   \]
   where $\mbox{Artin}(\beta)$ is Artin's automorphism on the free symmetric quandle.
\end{theorem}

We remark that the knot group of $F$ has a similar presentation (Remark~\ref{Remark: pres of knot group}).

The \textit{plat index} (\cite{Yasuda21}) of a surface-link $F$, denoted by $\mbox{Plat}(F)$, is defined as the half of the minimum number of the degree of adequate braided surfaces whose plat closures are equivalent to $F$.
By Theorem~\ref{MainTheoremA}, we obtain the following inequality between the plat index of $F$ and the symmetric quandle coloring number of $F$.

\begin{theorem}\label{Theorem: inequality for plat index}
   For any surface-link $F$ and any finite symmetric quandle $X$, the following inequality holds:
   \[
      \log_{(\# X)}\mathrm{col}_{X}(F) \leq \mathrm{Plat}(F),
   \]
   where $\# X$ is the order of $X$ and $\mbox{col}_X(F)$ is the $X$-coloring number of $F$.
\end{theorem}

It is known that the plat index of a surface-link $F$ is 1 if and only if $F$ is either a trivial 2-knot or a trivial non-orientable surface-knot (\cite{Yasuda21}).
Using Theorems~\ref{MainTheoremA} and \ref{Theorem: inequality for plat index}, we show the following theorem.

\begin{theorem}\label{MainTheoremB}
   For any integers $g\geq 0$ and $m \geq 2$, there exist infinitely many distinct orientable surface-knots of genus $g$ whose plat indices are $m$.
\end{theorem}

% Theorem~\ref{MainTheoremB} is proved by using this inequality.
% Theorems~\ref{Theorem: inequality for plat index} and \ref{MainTheoremB} are proved in Section~\ref{Section: Proof of main theorem B}.

This paper is organized as follows:
In section~\ref{Section: Plat form presentation}, we introduce a braided surface, a braid system, and a plat form presentation for a surface-link.
In section~\ref{Section: Quandles and symmetric quandles}, we review a quandle and a symmetric quandle.
We prove Theorem~\ref{MainTheoremA} in section~\ref{Section: Proof of main theorem A}, and we prove Theorem~\ref{MainTheoremB} in section~\ref{Section: Proof of main theorem B}.

\section{A plat form presentation}\label{Section: Plat form presentation}
\subsection{A plat form presentation for classical links}
Throughout this paper, we work in the PL category or the smooth category and assume that surface-links are locally flat in the PL category.
Let $n$ be a positive integer, $I = [0,1]$ be the interval, $D^2$ be the square $I^2$ or a $2$-disk in $\R^2$.
Let $Q_n$ be the subset of $n$ points in $D^2$ consisting of $q_k=\left({1}/{2},  {k}/{(n+1)} \right) \in I\times I =D^2$ for $k=1,2,\cdots, n$.

The \textit{braid group} $B_n$ is  the fundamental group $\pi_1(\mathcal{C}_n, Q_n)$ of the configuration space $\mathcal{C}_n$ of $n$ points of $\Int D^2$ with base point $Q_n$.
An element of $B_n$ is called an \textit{$n$-braid}.
A \textit{geometric $n$-braid} is a union of $n$ intervals $\beta$ embedded in $D^2\times I$ such that each component meets every open disk $\Int D^2\times \{t\}$ ($t\in I$) transversely at a single point, and $\partial \beta = Q_n\times \{0,1\}$.
We identify an $n$-braid with an equivalence class of a geometric $n$-braid.
% For a loop $f: (I, \partial I) \to (\mathcal{C}_n, Q_n)$, the geometric $n$-braid $\beta_f$ is defined by $\beta_f = \bigcup_{t\in I} |\partial f(t)|\times \{t\} \subset D^2\times I$.
% By this way, we identify a geometric $n$-braid with a loop $f: (I, \partial I) \to (\mathcal{C}_n, Q_n)$.
We denote by $\sigma_1, \sigma_2, \dots, \sigma_{n-1}$ the standard generators of $B_n$ or their representatives due to Artin (\cite{Artin1925}).

% An \textit{$n$-braid} is a union of $n$ intervals $\beta$ embedded in $D^2\times I$ such that each component meets every open disk $\Int D^2\times \{t\}$ ($t\in I$) transversely at a single point, and $\partial \beta = Q_n\times \{0,1\}$.
% The $n$-braid group $B_n$ is the group consisting of equivalence classes of $n$-braids in $D^2\times I$.
% We identify the braid group $B_n$ with the fundamental group $\pi_1(\mathcal{C}_n, Q_n)$ of the configuration space $\mathcal{C}_n$ of $n$ points of $\Int D^2$ with base point $Q_n$.
% Denote by $\sigma_1, \sigma_2, \dots, \sigma_{n-1}$ the standard generators of $B_n$ or their representatives due to Artin (\cite{Artin1925}).

% In order to introduce the plat closure for a braided surface, we define the plat closure of a geometric braid using wickets.

\begin{definition}[\cite{Brendle-Hatcher2008}]
   A \textit{wicket} is a semicircle in $D^2 \times I$ that meets $D^2 \times \{0\}$ orthogonally at its endpoints in $\Int D^2$.
   A \textit{configuration of $m$ wickets} is a disjoint union of $m$ wickets in $D^2\times I$.
   % The \textit{space of $m$ wickets} $\mathcal{W}_m$ is the space consists of all configurations of $m$ wickets.
\end{definition}

The \textit{enhanced boundary} of a configuration of $m$ wickets $w$, denoted by $\partial w$, is the boundary of $w$, which is $2m$ points of $\Int D^2$, equipped with a partition into $m$ pairs of points such that each pair of points is the boundary of a wicket of $w$.
$|\partial w|$ denotes the $2m$ points $\partial w$ forgetting the partition.
Then two configurations $w$, $w'$ are the same if they have the same enhanced boundary, $\partial w = \partial w'$.

The set $Q_{2m}$ equipped with the partition $\{ \{q_1, q_2\}, \dots, \{q_{2m-1}, q_{2m}\} \}$ bounds a unique configuration of $m$ wickets, which we call the \textit{standard configuration of $m$ wickets} and denote it by $w_0$.

Let $\beta$ be a geometric $2m$-braid in $D^2 \times I \subset \R^2 \times \R = \R^3$.
The \textit{plat closure} of $\beta$, denoted by $\widetilde{\beta}$, is a link obtained by attaching two copies of the standard configurations of $m$ wickets to $\beta$ as shown in Figure~\ref{Figure: A plat form of trefoil}.
Similarly, a link $L$ is in a \textit{plat form} if $L = \widetilde{\beta}$ for some $2m$-braid $\beta$.
Every link is equivalent to a link in a plat form.

\begin{figure}[h]
   \centering
   \includegraphics[height=15mm]{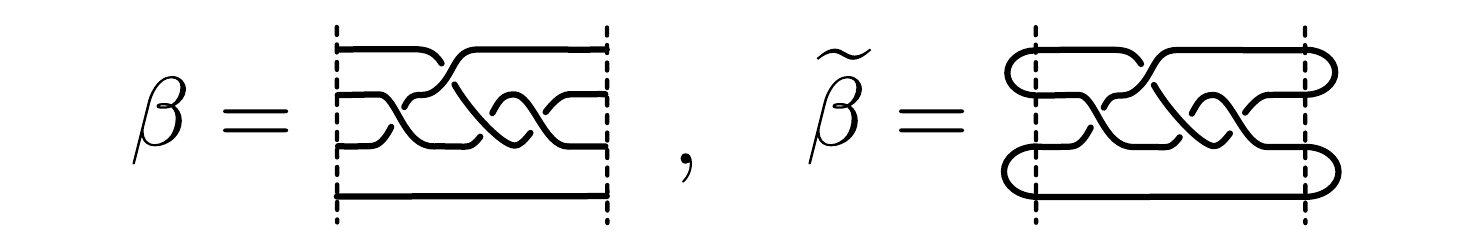}
   \caption{The plat closure of a braid.}
   \label{Figure: A plat form of trefoil}
\end{figure}

% \subsection{Surface-links}
% A \textit{surface-link} is a closed surface embedded in $\R^4$, and a \textit{surface-knot} is a connected surface-link.
% A \textit{$2$-knot} is a surface-knot homeomorphic to a $2$-sphere.
% A \textit{$2$-link} is a surface-link consisting of $2$-spheres.
% Two surface-links $F$ and $F'$ are said to be \textit{equivalent} if they are ambiently isotopic in $\R^4$.
% We denote this by $F \cong F'$.

% Let $h:\R^3\times \R^1 \to \R^1$ be the second factor projection.
% Set $F_{[t]} = F \cap \R^3\times \{t\}$ for $t\in \R$ ($F_{[t]}$ is called the \textit{cross-section} of $F$ at $t$).
% A \textit{motion picture} of $F$ is a $1$-parameter family $\{F_{[t]}\}_{t\in \R}$.
% We often describe surface-links using motion pictures.

% A surface-knot is \textit{trivial} if it is equivalent to a connected sum of standardly embedded $2$-spheres, tori, and projective planes (\cite{Hosokawa-Kawauchi1979}).
% Here standardly embedded projective planes $P_+$ and $P_-$ are illustrated in Figure~\ref{Figure: mp of P2}.

% \begin{figure}[h]
%    \centering
%    \includegraphics[height=7em]{img/img-mp-P2.pdf}
%    \caption{Motion pictures of $P_+$ and $P_-$.}
%    \label{Figure: mp of P2}
% \end{figure}

\subsection{Braided surfaces and braid systems}\label{subsection: Braided surfaces and braid systems}

Let $D_1$ and $D_2$ be the squares $I^2$ in $\R^2$ and let $\mathrm{pr}_i : D_1\times D_2 \to D_i$ $(i=1,2)$ be the $i$-th factor projection.
Let $y_0$ be a fixed base point of $D_2$ with $y_0 \in \partial D_2$.

\begin{definition}[\cite{Rudolph1983}, \cite{Viro90}]
   A surface $S$ embedded in $D_1\times D_2$ is a (\textit{pointed}) \textit{braided surface} of degree $n$ if $S$ satisfies the following conditions:
   \begin{enumerate}
      \item $\pi_S = \mathrm{pr}_2|_S: S \to D_2$ is a branched covering map of degree $n$.
      \item $\partial S$ is the closure of an $n$-braid in the solid torus $D_1\times \partial D_2$.
      \item $\mathrm{pr_1}(\pi_S^{-1}(y_0))=Q_n$.
   \end{enumerate}
   A \textit{$2$-dimensional braid} of degree $n$ is a braided surface $S$ of degree $n$ such that $\mathrm{pr_1}(\pi_S^{-1}(y))=Q_n$ for all $y \in \partial D_2$.
\end{definition}

Throughout this paper, we assume that every braided surface of degree $n$ is \textit{simple}, that is, the preimage of each branch locus of $\pi_S$ consists of $n-1$ points.
An orientation of $S$ is \textit{compatible} with an orientation of $D_2$ if for each regular point $x \in S$ of $\pi_S$, the orientation of $S$ at $x$ is coherent with the orientation of $D_2$ at $\pi_S(x)$ via $\pi_S$.
Unless otherwise stated, an orientation of $S$ is always chosen to be compatible with an orientation of $D_2$.

For two braided surfaces of the same degree, they are said to be \textit{equivalent} if they are ambiently isotopic by an isotopy $\{h_s\}_{s\in I}$ of $D_1\times D_2$ such that each $h_s$ ($s\in I$) is fiber-preserving when we regard $D_1 \times D_2$ as the trivial $D_1$-bundle over $D_2$, and the restriction of $h_s$ to ${\mathrm{pr}_2^{-1}(y_0)}$ is the identity map.

Next, we recall the braid monodromy and a braid system of a braided surface $S$.
Let $\Sigma(S) \subset D_2$ be the branch locus of $\pi_S$.
% The \textit{braid monodromy} of $S$ is a homomorphism $\rho_S: \pi_1(D_2\setminus \Sigma(S), y_0) \to B_n = \pi_1(\mathcal{C}_n, Q_n)$ defined as follows:
For a loop $c: (I,\partial I) \to (D_2\setminus \Sigma(S), y_0)$, we define a loop $\rho_S(c): (I,\partial I) \to (\mathcal{C}_n, Q_n)$ as $\rho_S(c)(t) = \mathrm{pr}_1(\pi_S(c(t))^{-1})$.
Then this map induces a homomorphism $\rho_S: \pi_1(D_2\setminus \Sigma(S), y_0) \to B_n = \pi_1(\mathcal{C}_n, Q_n)$ by $\rho_S([c]) = [\rho_S(c)]$, which is called the \textit{braid monodromy} of $S$.
% we define the braid monodromy of $S$ as a group homomorphism sending $[c]$ to $\rho_S([c]) = [\rho_S(c)] \in \pi_1(\mathcal{C}_n, Q_n)$.

Let $r\geq 1$ be a positive integer.
A \textit{Hurwitz arc system} in $D_2$ (with the base point $y_0$) is an $r$-tuple $\mathcal{A} = (\alpha_1, \cdots, \alpha_r)$ of oriented simple arcs in $D_2$ such that
\begin{enumerate}
   % \item the initial point of $\alpha_i$ is a point in $\Sigma(S)$,
   \item $\alpha_i \cap \partial D_2 = \partial \alpha_i \cap \partial D_2 = \{y_0\}$ and this is the terminal point of $\alpha_i$,
   \item $\alpha_i \cap \alpha_j = \{y_0\}$ ($i\neq j$), and
   \item $\alpha_1, \ldots,  \alpha_r$ appears in this order around the base point $y_0$.
\end{enumerate}
The \textit{starting point set} of $\mathcal{A}$ is the set of initial points of $\alpha_1, \ldots, \alpha_r$.
% Two Hurwitz arc systems are said to be \textit{equivalent} if they are isotopic in $D_2$ rel $\partial D_2 \cup \Sigma(S)$.

Let $\mathcal{A} = (\alpha_1, \cdots, \alpha_r)$ be a Hurwitz arc system with the starting point set $\Sigma(S)$ and $N_i$ be a (small) regular neighborhood of the starting point of $\alpha_i$ and let $\overline{\alpha_i}$ be an oriented arc obtained from $\alpha_i$ by restricting to $D_2\setminus \Int N_i$.
For each $\alpha_i$, let $\gamma_i$ be a loop $\overline{\alpha_i}^{-1} \cdot \partial N_i \cdot \overline{\alpha_i}$ in $D_2\setminus \Sigma(S)$ with base point $y_0$, where $\partial N_i$ is oriented counter-clockwise.
Then $\pi_1(D_2\setminus \Sigma(S), y_0)$ is the free group generated by $[\gamma_1], [\gamma_2], \ldots, [\gamma_r]$.
% Note that $[\partial D_2]$ $= [\gamma_1][\gamma_2] \cdots [\gamma_r]$.

\begin{definition}[\cite{Rudolph1983,Kamada2002_book}]\label{Definition: braid system}
   A \textit{braid system} of $S$ is an $r$-tuple $b_S$ of $n$-braids defined as
   \[
      b_S ~=~ (\rho_S([\gamma_1]) \ldots, \rho_S([\gamma_r])).
   \]
\end{definition}

It is known that $\rho_S([\gamma_i])$ is a conjugation of some generator $\sigma_{k_i}^{\varepsilon_i}$, where $\varepsilon_i \in \{\pm 1\}$ is called the \textit{sign} of the branch point whose projection is the starting point of $\alpha_i$.

We do not mention a \textit{chart description} of $S$ in this paper.
However, it is worth noting that a braid system of $S$ can be easily obtained from its chart description, see \cite{Kamada2002_book} for details.

Let $G$ be a group and $G^r$ the set of $r$-tuples of elements of $G$.
The \textit{slide action} of the braid group $B_r$ on $G^r$ is a left group action defined by
\[
   \mathrm{slide}(\sigma_j)(g_1, \ldots, g_r) ~=~ (g_1, \ldots, g_{j-1}, g_jg_{j+1}g_j^{-1}, g_j, g_{j+2}, \ldots, g_r)
\]
for $(g_1, \ldots, g_r) \in G^r$ and $\sigma_j \in B_r$.
% We remark that this is a left action so that $\mathrm{slide}(\sigma_i\sigma_j)(g_1, \ldots, g_r) = \mathrm{slide}(\sigma_i)(\mathrm{slide}(\sigma_j)(g_1, \ldots, g_r))$.
Two elements of $G^r$ are said to be \textit{Hurwitz equivalent} (or \textit{slide equivalent}) if their orbits of slide action are the same.

It is known that a Hurwitz equivalence class of a braid system of $S$ is uniquely determined by the equivalence class of $S$(cf. \cite{Kamada2002_book, Moishezon1981, Rudolph1985}) although a braid system of $S$ depends on the choice of a Hurwitz arc system $\mathcal{A}$.

\subsection{A plat form presentation for surface-links}\label{subsection: A plat form for surface-links}
% In this subsection, we introduce the plat closure of adequate braided surfaces and a plat form presentation for surface-links.
For an integer $m \geq 1$, $\mathcal{W}_m$ denotes the space consisting of all configurations of $m$ wickets.
% The fundamental group $\pi_1(\mathcal{W}_m, w_0)$ is called the \textit{wicket group} in \cite{Brendle-Hatcher2008}.
% Let $|\partial|: (\mathcal{W}_m, w_0) \to (\mathcal{C}_{2m}, Q_{2m})$ be the continuous map sending $w$ to $|\partial w|$, and $|\partial|_\ast: \pi_1(\mathcal{W}_m, w_0) \to \pi_1(\mathcal{C}_{2m}, Q_{2m}) =B_{2m}$ be the induced homomorphism.
For a loop $f: (I, \partial I) \to (\mathcal{W}_m, w_0)$, the geometric $2m$-braid $\beta_f$ is defined as
\[
   \beta_f ~=~ \bigcup_{t\in I} |\partial f(t)|\times \{t\} \subset D\times I.
\]

\begin{definition}
   A geometric $2m$-braid $\beta$ is \textit{adequate} if $\beta = \beta_f$ for some loop $f: (I, \partial I)\to (\mathcal{W}_m, w_0)$.
\end{definition}

Since a configuration of wickets is unique for its enhanced boundary, two loops $f,g: (I, \partial I) \to (\mathcal{W}_m, w_0)$ are the same if $\beta_f = \beta_g$.

\textit{Hilden's subgroup} $K_{2m}$ is defined as the subgroup of $B_{2m}$ generated by the following elements (\cite{Hilden1975}, cf. \cite{Birman1976}):
\[
   % K_{2m} ~=~ \left\langle \sigma_1,~ \sigma_2\sigma_1\sigma_3\sigma_2,~ \sigma_{2i} \sigma_{2i-1} \sigma_{2i+1}^{-1} \sigma_{2i}^{-1} ~|~ i = 1, \dots, m-1 \right\rangle.
   K_{2m} ~=~ \left\langle
   \begin{array}{c|c}
      \sigma_1,~ \sigma_2\sigma_1\sigma_3\sigma_2,~ \sigma_{2j} \sigma_{2j-1} \sigma_{2j+1}^{-1} \sigma_{2j}^{-1} &
      j = 1, \dots, m-1
   \end{array}
 \right\rangle.
\]
% Then, for each $m \geq 1$, $|\partial|_\ast$ is injective and the image is Hilden's subgroup $K_{2m}$ (\cite{Brendle-Hatcher2008}), that is, $\pi_1(\mathcal{W}_m, w_0)$ is isomorphic to $K_{2m}$.
% restated as follows:
% Let $f: (I, \partial I)\to (\mathcal{W}_m, w_0)$ be a loop.
% Consider a $2m$-braid $\beta_f = \bigcup_{t\in I} |\partial f(t)|\times \{t\} \subset D\times I$, then the isomorphism  sends $[f]\in \pi_1(\mathcal{W}_m, w_0)$ to $ [\beta_f] \in K_{2m}$.

Two $2m$-braids $\sigma_{2j}\sigma_{2j-1}\sigma_{2j+1}\sigma_{2j}$ and $\sigma_{2j}\sigma_{2j-1}\sigma_{2j+1}^{-1}\sigma_{2j}^{-1}$ are denoted by $\tau_j$ and $\upsilon_j$, respectively.
We remark that $\sigma_{2i-1}, \tau_j, \upsilon_j \in K_{2m}$ for $i = 1,\ldots, m$ and $j = 1, \ldots, m-1$.
% Denote $\tau_j = \sigma_{2j}\sigma_{2j-1}\sigma_{2j+1}\sigma_{2j}$, and $\upsilon_j = \sigma_{2j}\sigma_{2j-1}\sigma_{2j+1}^{-1}\sigma_{2j}^{-1}$ for $1 \leq j \leq m-1$.

\begin{figure}[h]
   \centering
   \includegraphics[height = 15mm]{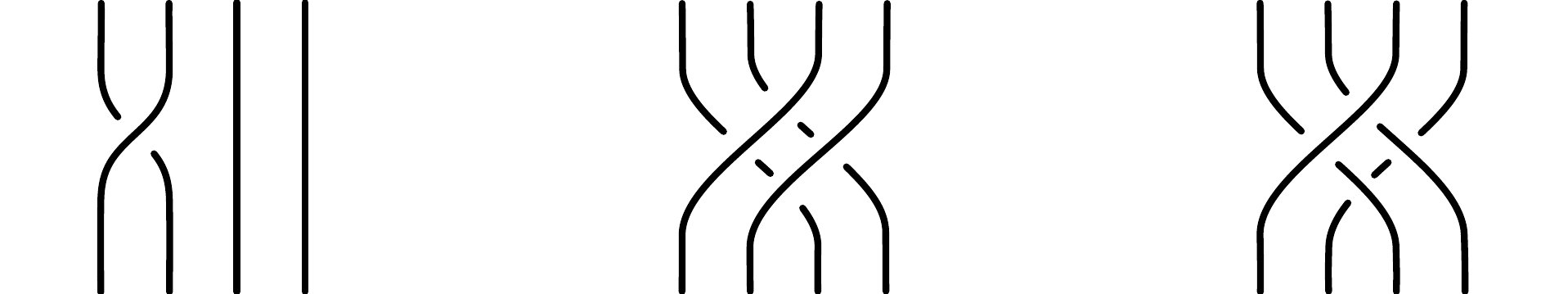}
   \caption{Generators of Hilden's subgroup $K_4$ ($m = 2$)}
   \label{Figure: braids of Hilden subgroup}
\end{figure}

Brendle and Hatcher \cite{Brendle-Hatcher2008} showed that $\pi_1(\mathcal{W}_m, w_0)$ is isomorphic to $K_{2m}$.
Such an isomorphism is given by sending $[f] \in \pi_1(\mathcal{W}_m, w_0)$ to $[\beta_f] \in B_{2m}$.
We remark that Hilden's subgroup $K_{2m}$ consists of elements of $B_{2m}$ represented by some adequate $2m$-braid.

We fix a loop $\mu: (I, \partial I) \to (\partial D_2, y_0)$ which runs once on $\partial D_2$ counter-clockwise.
For a braided surface $S$ of degree $n$, let $\beta_S$ be the geometric $n$-braid in $D_1\times I$ defined by
\[
   \beta_S ~=~ \bigcup_{t\in I} \mathrm{pr}_1(\pi_S^{-1}(\mu(t)))\times \{t\},
\]
where $\pi_S: S \to D_2$ is the simple branched covering map induced from $S$.
Then $\partial S$ is the closure of $\beta_S$ in $D_1\times (I/\partial I) \cong D_1\times \partial D_2$.

\begin{definition}\label{Def:adequate}
   A braided surface $S$ in $D_1 \times D_2$ is \textit{adequate} if $\beta_S$ is adequate.
\end{definition}

The degree of an adequate braided surface is even and every 2-dimensional braid of even degree is adequate.
By definition, a loop $f: (I, \partial I) \to (\mathcal{W}_m, w_0)$ satisfying $\beta_S = \beta_f$ is unique for $S$, thus we denote it by $f_S$.

We assume $D_1\times D_2 \subset \R^2 \times \R^2 = \R^4$.
% Let $y_0 \in \partial D_2$ be the base point when we consider a braided surface.
Let $N$ be a regular neighborhood of $\partial D_2$ in $\R^2\setminus \Int{D_2}$, which is homeomorphic to an annulus $I\times S^1$.
We identify $N$ with $I\times S^1$ by a fixed identification map $\phi: I \times S^1 \to N$ such that $\phi(0,p(t)) = \mu(t) \in \partial D_2$ for all $t \in I$, where $p: I \to S^1 = I/\partial I$ is the quotient map.

\begin{definition}
   Let $S$ be an adequate braided surface of degree $2m$.
   The surface \textit{of wicket type associated with $S$} is a properly embedded surface $A_S$ in $D_1\times N$ defined by
   \[
      A_S ~=~ \bigcup_{t\in I} f_S(t)\times \{p(t)\} ~\subset~ (D\times I) \times S^1 ~=~ D_1\times N.
   \]
\end{definition}

The surface $A_S$ is uniquely determined by $S$ since the loop $f_S$ is unique for $S$.
Furthermore, the intersection of $S$ with $A_S$ is equal to their boundaries; $S \cap A_S = \partial S = \partial A_S$.
Thus, we obtain a surface-link from an adequate braided surface $S$ by taking the union $S \cup A_S$.
% A surface $A$ of wicket type is a union of annuli or M\"{o}bius bands, and that $\partial A= \partial A_f$ is expressed in terms of $|\partial| \circ f: (I, \partial I) \to (\mathcal{C}_{2m}, Q_{2m})$ as follows:
% \[
%    \partial A ~=~ \bigcup_{t\in I} |\partial f(t)|\times \{p(t)\} ~\subset~ D\times S^1 ~=~ D_1 \times \partial D^2.
% \]
% Two surfaces $A$ and $A'$ of wicket type with $\partial A = \partial A'$ are identical since two loops $f$ and $f'$ in $(\mathcal{W}_m, w_0)$ with $|\partial| \circ f = |\partial| \circ f'$ are identical.

% Let $S$ be an adequate braided surface, and let $f: (I, \partial I) \to (\mathcal{W}_m, w_0)$ be a loop with $g_S = |\partial| \circ f$.
% Then $S \cap A_f = \partial S = \partial A_f$.
% We denote $A_f$ by $A_S$ and say that $A_S$ is \textit{associated with $S$}.

\begin{definition}
   Let $S$ be an adequate braided surface, and $A_S$ be the surface of wicket type associated with $S$.
   % The union $\widetilde{S} = S \cup A_S$ is called the \textit{plat closure} of $S$.
   The \textit{plat closure} of $S$, denoted by $\widetilde{S}$, is the union $S \cup A_S$ in $\R^4$.
\end{definition}

\begin{definition}
A surface-link is said to be \textit{in a plat form} if it is the plat closure of some adequate braided surface.
Moreover, a surface-link is said to be \textit{in a genuine plat form} if it is the plat closure of some $2$-dimensional braid.
\end{definition}

In \cite{Yasuda21}, it is shown that every surface-link has a plat form presentation and that every orientable surface-link has a genuine plat form presentation.

\begin{definition}
   Let $F$ be a surface-link.
   The \textit{plat index} of $F$ is defined as the half of the minimum degree of adequate braided surfaces whose plat closures are equivalent to $F$, denoted by $\mbox{Plat}(F)$.
   Similarly, the \textit{genuine plat index} of $F$ is defined as the half of the minimum degree of 2-dimensional braids of even degree whose plat closures are equivalent to $F$, denoted by $\mbox{g.Plat}(F)$.
\end{definition}

\section{Symmetric quandles and symmetric quandle presentations}\label{Section: Quandles and symmetric quandles}

In this section, we introduce the notions of symmetric quandles and their presentation.

\begin{definition}[\cite{Joyce82, Mateev84}]
   A \textit{quandle} is a nonempty set $X$ with a binary operation $(x,y) \mapsto x^y $ that satisfies the following three conditions:
   \begin{itemize}
      \item[(Q1)] $x^x = x$ for any $x\in X$,
      \item[(Q2)] there is a unique binary operation $(x,y) \mapsto x^{y^{-1}}$ such that $(x^y)^{y^{-1}} = (x^{y^{-1}})^y = x$ for any $x,y \in X$, and
      \item[(Q3)] $(x^y)^z = (x^z)^{(x^y)}$ for any $x,y,z\in X$.
   \end{itemize}
   A \textit{rack} is a nonempty set with a binary operation that satisfies (Q2) and (Q3).
\end{definition}

\begin{definition}[\cite{Kamada2006, Kamada-Oshiro2010}]
   Let $X$ be a quandle.
   An involution map $\rho: X \to X$ is called a \textit{good involution} if $\rho$ satisfies the following two conditions:
   \begin{itemize}
      \item[(SQ1)] $\rho(x^y) = \rho(x)^y$ for any $x,y \in X$, and
      \item[(SQ2)] $x^{\rho(y)} = x^{y^{-1}}$ for any $x,y \in X$.
   \end{itemize}
   A \textit{symmetric quandle} is a pair $(X, \rho)$ of a quandle $X$ and its good involution $\rho$.
\end{definition}

For symmetric quandles $(X,\rho)$ and $(X', \rho')$, a map $f: X \to X'$ is a \textit{symmetric quandle homomorphism} if $f$ satisfies $f(x^y) = f(x)^{f(y)}$ and $f(\rho(x)) = \rho'(f(x))$ for any $x,y\in X$.

We write $x^{-1}$ and $x^{-y}$ for $\rho(x)$ and $\rho(x^y)$, respectively, when the good involution of $(X,\rho)$ is clearly.
Then, $(x^y)^{-1} = (x^{-1})^y = x^{-y}$ and $x^{(y^{-1})} = x^{y^{-1}}$ hold for $x,y \in X$.

\begin{example}
   The \textit{dihedral quandle} $R_p$ of order $p$ is the cyclic group of order $p$ with a binary operator defined by $a^b = 2b-a$ for $a, b \in R_p$.
   % Let $R_p$ be the cyclic group $\Z/p\Z$ of order $p$ with a binary operator defined by $a^b = 2b-a$ ($a,b\in \Z/p\Z$).
   % Then $R_p$ is a quandle, which is called \textit{dihedral quandle} of order $p$.
   For every positive integer $p$, the identity map $\id_{R_p}: R_p \to R_p$ is a good involution of $R_p$.
   Thus, $(R_n, \id_{R_n})$ is a symmetric quandle.

   In general, a \textit{kei} (or an \textit{involutive quandle}) is a quandle $X$ such that $(x^y)^{y} = x$ for any $x,y\in X$.
   It is known that the identity map $\id_X$ on a quandle $X$ is a good involution if and only if $X$ is a kei (\cite{Kamada-Oshiro2010}).
\end{example}

\begin{example}[\cite{Kamada-Oshiro2010}]
   Let $K$ be a properly embedded $n$-submanifold in $(n+2)$-manifold $W$, $N(K)$ be a regular neighborhood of $K$ in $W$, $E(K) = \mathrm{cl}(W\setminus N(K))$ be the exterior of $K$, and $*$ be a fixed based point of $E(K)$.
   A \textit{noose} of $K$ is a pair $(D,\alpha)$ of an oriented meridional disk $D$ of $K$ and an oriented arc $\alpha$ in $E(K)$ connecting from a point of $\partial D$ to $*$.
   The \textit{knot full quandle} (\cite{Kamada2006,Kamada-Oshiro2010}) of $(W, K)$ is a quandle consisting of homotopy classes $[(D,\alpha)]$ of nooses of $K$ with a binary operator defined by
   \[
      [(D_1, \alpha_1)]^{[(D_2, \alpha_2)]} = [(D_1, \alpha_1\cdot\alpha_2^{-1}\cdot \partial D_2 \cdot\alpha_2)].
   \]
   Note that the knot full quandle of $K$ is independent of the choice of the based point $*$ when $W$ is connected.
   Thus we denote it by $\widetilde{Q}(W,K)$ or simply $\widetilde{Q}(K)$.
   % Thus we denote it by $\widetilde{Q}(K)$ or $\widetilde{Q}(W, K)$.
   Here, "full" means that $\widetilde{Q}(K)$ contains both nooses of $(D, \alpha)$ and $(-D,\alpha)$, where $-D$ is $D$ with the reverse orientation.
   Let $\rho$ be an involution map of $\widetilde{Q}(K)$ defined by $\rho([(D,\alpha)]) = [(-D,\alpha)]$, then $\rho$ is a good involution.
   We call $(\widetilde{Q}(K), \rho)$ the \textit{knot symmetric quandle} of $K$, denoted by $X(W, K)$ or simply $X(K)$. % (\cite{Kamada2006,Kamada-Oshiro2010}).

   For a pair $(W, K)$ and a symmetric quandle $X$, let $\mbox{Col}_X(F)$ be the set of symmetric quandle homomorphisms from $X(W, K)$ to $X$.
   The \textit{$X$-coloring number} of $(W, K)$ is defined as the cardinal number of $\mbox{Col}_X(F)$, denoted by $\mbox{col}_{X}(W, K)$ or $\mbox{col}_X(K)$.
\end{example}

\begin{example}[\cite{Kamada2014}]
   Let $A$ be a non-empty set and $F(A)$ be the free group on $A$.
   We put $\overline{A} = \{ a^{-1}\in F(A) ~|~ a\in A \}$ and suppose that $A$ and $\overline{A}$ are disjoint.
   Then $\widetilde{\mbox{FR}}(A) = (A\cup \overline{A})\times F(A)$ is a rack equipped with a binary operator given by $(a,w)^{(b,z)} = (a, wz^{-1}bz)$ for $(a,w), (b,z) \in \widetilde{\mbox{FR}}(A)$.

   Let $\sim_{sq}$ be the equivalence relation on $\widetilde{\mbox{FR}}(A)$ generated by $(a,w) \sim_{sq} (a,a^{aw})$ for $a\in (A\cup \overline{A})$ and $w\in F(A)$.
   Then the rack operation on $\widetilde{\mbox{FR}}(A)$ induces a quandle operation on $\widetilde{\mbox{FQ}}(A) = \widetilde{\mbox{FR}}(A)/\sim_{sq}$.
   An involution map $\rho_A: \widetilde{\mbox{FQ}}(A) \to \widetilde{\mbox{FQ}}(A)$ given by $\rho(a^{\pm1}, w) \mapsto (a^{\mp1}, w)$ is a good involution of $\widetilde{\mbox{FQ}}(A)$.
   $(\widetilde{\mbox{FQ}}(A), \rho_A)$ is called the \textit{free symmetric quandle} on $A$, denoted by $\mbox{FSQ}(A)$.
   For $a \in A$ and $w\in F(A)$, elements $(a,w)$ and $(a^{-1}, w)$ in $\mbox{FSQ}(A)$ are denoted by $a^w$ and $a^{-w}$, respectively.
   For $a\in (A \cup \overline{A})$ and the identity $1 \in F(A)$, we simply write $a$ for $a^1 \in \mbox{FSQ}(A)$.
\end{example}

Let $R$ be a subset of $\mbox{FSQ}(A)\times \mbox{FSQ}(A)$.
We denote $R_0$ as $R$.
For $n\geq 1$, we expand $R_{n-1}$ to $R_n$ by the following moves:
\begin{enumerate}
   \item[(E1)] For every $x\in FQ(A)$, add $(x,x)\in R_n$.
   \item[(E2)] For every $(x,y)\in R_{n-1}$, add $(y,x)\in R_n$.
   \item[(E3)] For every $(x,y), (y,z) \in R_{n-1}$, add $(x,z)\in R_n$.
   \item[(R1)] For every $(x,y)\in R_{n-1}$ and $a\in A$, add $(x^a,y^a), (x^{-a},y^{-a})\in R_n$.
   \item[(R2)] For every $(x,y)\in R_{n-1}$ and $t\in \mbox{FQ}(A)$, add $(t^x,t^y), (t^{-x},t^{-y})\in R_n$.
   \item[(Q)] For every $a\in A$, add $(a^a,a) \in R_n$.
\end{enumerate}

Then the union $\langle\langle R \rangle\rangle_{sq} = \bigcup_{n \geq 0} R_n$ is called the \textit{set of symmetric quandle consequences} of $R$, which is the smallest congruence containing $R$ with respect to the axioms of a symmetric quandle (\cite{Kamada2014}).
An element of $\langle\langle R \rangle\rangle_{sq}$ is called a \textit{consequence} of $R$.
Then, the quandle operation and the good involution on $\mbox{FSQ}(A)$ induce a quandle operation and a good involution on $\mbox{FSQ}(A) / \langle\langle R \rangle\rangle_{sq}$.
This symmetric quandle is denoted by $\langle A ~|~ R\rangle_{sq}$.

\begin{definition}[\cite{Kamada2014}]
   The symmetric quandle $\langle A ~|~ R \rangle_{sq}$ is a \textit{symmetric quandle presentation} (or simply  \textit{ presentation}) of $X$ if $X$ is isomorphic to $\langle A ~|~ R\rangle_{sq}$.
\end{definition}

Each element $a\in A$ is called a \textit{generator} of $X$ and each element $r\in R$ is called a \textit{relator} of $X$.
A presentations $\langle A ~|~ R \rangle_{sq}$ is \textit{finite} if both $A$ and $R$ are finite.

We usually write $(a^w, b^w) \in \langle\langle R \rangle\rangle_{sq}$ as an equation $a^w = b^z$.

The following lemmas are proved for racks and quandles (cf. \cite{Kamada2017_book}), and these proofs can be extended to the case of symmetric quandles.

\begin{lemma}[{cf. \cite[Proposition~8.6.1]{Kamada2017_book}}]\label{Lemma: Universal property for free symmetric quandle}
   Let $A$ be a non-empty set and $X$ be a symmetric quandle.
   For a map $f: A\to X$, there is a unique homomorphism $f_\sharp: \mbox{FSQ}(A) \to X$ such that $f = f_\sharp \circ \iota$, where $\iota: A \to \mbox{FSQ}(A)$ is the natural inclusion map.
\end{lemma}

\begin{lemma}[cf. {\cite[Lemma~8.6.3]{Kamada2017_book}}]\label{Lemma: Universal property for symmetric quandle presentation}
   Let $X$ be a symmetric quandle with a presentation $\langle A ~|~ R\rangle_{sq}$.
   For a symmetric quandle $Y$, a map $f: A \to Y$ extends to a homomorphism $f:X \to Y$ if and only if $f(a)^{f(w)} = f(b)^{f(z)}$ holds for $(a^w,b^z)\in R$.
\end{lemma}

\begin{definition}[cf. \cite{Fenn-Rourke1992}]
   Let $X$ be a symmetric quandle with a presentation $\langle A ~|~ R\rangle_{sq}$.
   The following four moves are called \textit{Tietze moves} on a symmetric quandle presentation.
   \begin{enumerate}
      \item[(T1)] Add a consequence of $R$ to $R$.
      \item[(T2)] Delete a consequence of other relators from $R$.
      \item[(T3)] Introduce a new generator $a$ and a new relator $(a,r)$ to $A$ and $R$, respectively, where $r$ is an element of $F(A)$.
      \item[(T4)] Delete a generator $a$ and a relator $(a,r)$ from $A$ and $R$, respectively, where $(a,r)$ is a consequence of other relators and $a$ does not occur in other relators in $R$.
   \end{enumerate}
\end{definition}

Note that (T2) and (T4) are inverse moves of (T1) and (T3), respectively.
Fenn and Rourke \cite{Fenn-Rourke1992} proved Tietze move theorem for racks, and their proof can be extended for symmetric quandles.

\begin{proposition}[cf. \cite{Fenn-Rourke1992}]\label{Theorem: Tietze theorem for symmetric quandle}
   Any two finite presentations of a symmetric quandle are related by a finite sequence of Tietze moves.
\end{proposition}

\section{Computation of the knot symmetric quandles}\label{Section: Proof of main theorem A}
In this section, we compute the knot symmetric quandle for surface-links using a plat form presentation.

\subsection{Diagrams and semi-sheets}

First, we introduce a diagram and a semi-sheet for an $n$-submanifold $K$ embedded in $\R^{n+2}$.
Without loss of generality, we may suppose that $K$ is in general position with respect to the projection $\pi: \R^{n+2} = \R \times \R^{n+1} \to \R^{n+1}$.
Let $\Delta_K$ be the closure of the singular set of $\pi(K)$;
\[
   \Delta_K = \mathrm{cl}(\{ x\in \pi(K) \in \R^{n+1} ~|~ \#(\pi^{-1}(x)\cap K) > 1 \}).
\]

A \textit{diagram} $D_K$ of $K$ is a union of $\mathrm{cl}(\pi(K)\setminus N(\Delta_K))$ and the projection of the upper preimage of $N(\Delta_K)$, where $N(\Delta_K)$ is a regular neighborhood of $\Delta_K$.
Figure~\ref{Figure: Local models of diagram} depicts local models of a surface-link diagram ($n=2$).

\begin{figure}[h]
   \centering
   \includegraphics[width = \hsize]{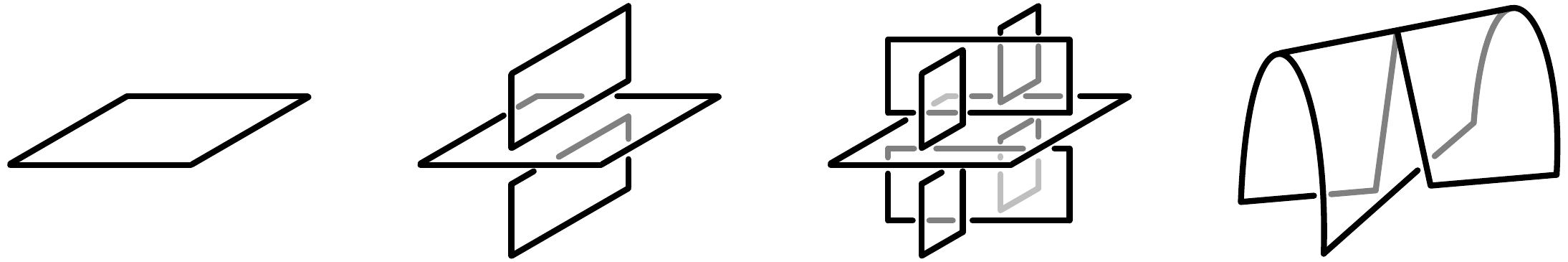}
   \caption{Local models of a diagram of a surface-link.}
   \label{Figure: Local models of diagram}
\end{figure}

Each connected component of $\mathrm{cl}(\pi(K)\setminus N(\Delta_K))$ is called a \textit{semi-sheet} of $D_K$.
Figure~\ref{Figure: Local models of semi-sheet} depicts local models of semi-sheets of a surface-link diagram ($n = 2$).
We also call it a \textit{semi-arc} when $K$ is a classical link.

\begin{figure}[h]
   \centering
   \includegraphics[width = \hsize]{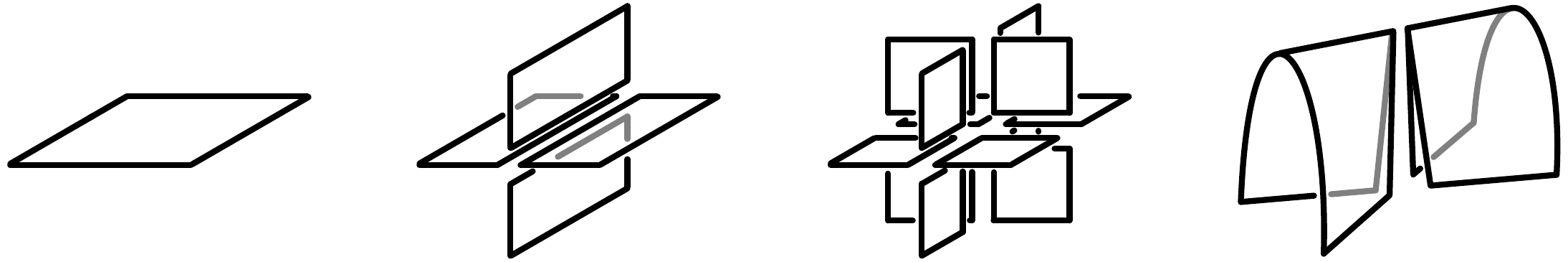}
   \caption{Local models of semi-sheets of a surface-link.}
   \label{Figure: Local models of semi-sheet}
\end{figure}

When we regard $\Delta_K$ as a stratified complex, each $(n-1)$-dimensional stratum of $\Delta_K$ consists of transverse double points (cf. \cite{Rourke-Sanderson1982}), which is called the \textit{double point stratum} of $D_K$.
Each double point stratum of a classical link diagram corresponds to a crossing, and each double point stratum of a surface-link diagram corresponds to a double point curve.

Let $x_1, \ldots, x_n$ be semi-sheets of $D_K$.
Since semi-sheets are orientable (\cite{Kamada2001}), we give them orientations arbitrarily.
For semi-sheets $x_i, x_j, x_s, x_t$ around a double point stratum of $D_K$ as in Figure~\ref{Figure: semi-sheet and its labeling}, we define the \textit{A-relation} and the \textit{B-relation} as follows:
Assume that $x_i$ and $x_j$ are upper semi-sheets and that $x_s$ and $x_t$ are lower semi-sheets (Figure~\ref{Figure: semi-sheet and its labeling}).
The A-relation is defined as $x_i = x_j$ if orientations of $x_i$ and $x_j$ are coherent (Figure~\ref{Figure: A, B-relations}-(a)) and $x_i = x_j^{-1}$ otherwise (Figure~\ref{Figure: A, B-relations}-(b)), where normal vectors on semi-sheets described in Figure~\ref{Figure: A, B-relations} represent their orientations (cf. \cite{Kamada2014}).
Suppose that the normal vector of $x_i$ is directed from $x_s$ to $x_t$.
The B-relation is defined as $x_t = x_s^{x_i}$ if orientations of $x_s$ and $x_t$ are coherent (Figure~\ref{Figure: A, B-relations}-(c)) and $x_t^{-1} = x_s^{x_i}$ otherwise (Figure~\ref{Figure: A, B-relations}-(d)).
The set of the A-relations and B-relations of $D_K$ is denoted by $\mbox{A-rel}_{D_K}$ and $\mbox{B-rel}_{D_K}$, respectively.

\begin{figure}[h]
   \centering
   \includegraphics[width = \hsize]{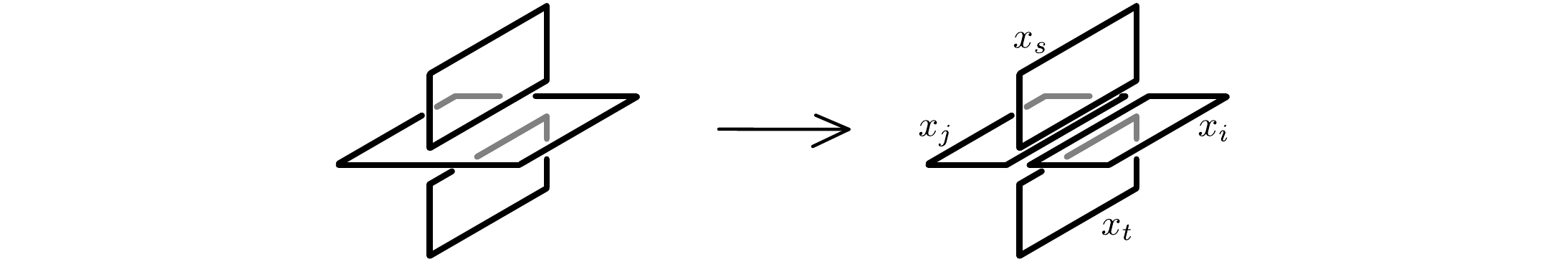}
   \caption{Semi-sheets around a double point strata and their labeling.}
   \label{Figure: semi-sheet and its labeling}
\end{figure}

\begin{figure}[h]
   \centering
   \includegraphics[width = \hsize]{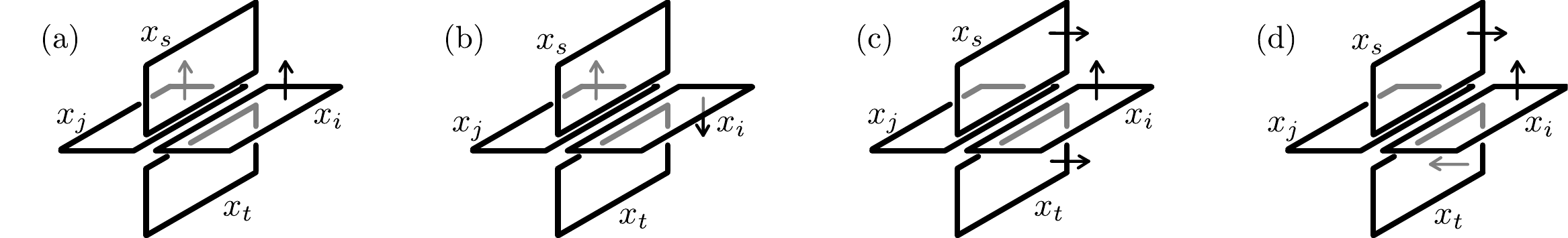}
   \caption{The A-relation and the B-relation.}
   \label{Figure: A, B-relations}
\end{figure}

Using these relations, we define a symmetric quandle $X(D_K)$ by
\[
   X(D_K) ~=~ \left\langle x_1, \ldots, x_{s}~
   \begin{array}{|c}
      \mbox{A-rel}_{D_K}, \mbox{B-rel}_{D_K}
   \end{array}
   \right\rangle_{\mathrm{sq}}.
\]
It is known that $X(D_K)$ is a well-defined symmetric quandle although the A-relation and the B-relation depend on the choice of orientations and labeling of semi-sheets (\cite{Kamada2014}).

\begin{proposition}[\cite{Kamada2014}]\label{Proposition: qdle pres. in 4-space}
   The knot symmetric quandle $X(K)$ is isomorphic to $X(D_K)$.
\end{proposition}

The isomorphism between $X(D_K)$ and $X(K)$ is given in \cite[Theorem~4.11]{Kamada2014} as follows:
Let $\Sigma = \mathrm{cl}(\pi(K)\setminus N(\Delta_K))$ be the union of semi-sheets of $D_K$.
We may assume that
\[
   K \cap [0,\infty) \times \R^{n+1} ~=~ [0,1] \times \partial \Sigma \cup \{1\} \times \Sigma.
\]
We take a base point $*$ of $\R^4$ such that the first coordinate of $*$ is sufficiently large so that for each $1\leq i \leq n$, we can take a noose $(D_i,\alpha_i)$ of $K$ such that $D_i$ is an oriented meridional disk on $\{1\}\times x_i$ and $\alpha_i$ is a line from a point of $\partial D_i$ to $*$.
Then, a homomorphism sending $x_i \in X(D_K)$ to $[(D_i,\alpha_i)] \in X(K)$ induces an isomorphism between them.

This argument can be extended to the case that $K$ is a properly embedded $n$-submanifold in $\R^{n+2}_+ = \R^{n+1}\times [0,\infty)$ as follows:
Let $\pi: \R^{n+2}_+ = \R \times \R^{n+1}_+ \to \R^{n+1}_+$ be the projection.
Then, $\Delta_K$, $D_K$, and $X(D_K)$ are defined similarly. % to the case of $\R^{n+2}$.
We may assume that
\[
   K \cap [0,\infty) \times \R^{n+1}_+ ~=~ [0,1] \times \partial \Sigma \cup \{1\} \times \Sigma,
\]
where $\Sigma = \pi(K) \setminus N(\Delta_K)$ is a union of semi-sheets of $D_K$.
We take a base point $*$ of $\R^{n+2}_+$ such that the first coordinate of $*$ is sufficiently large so that, for each semi-sheet $x_i$ of $D_K$, we can take a noose $(D_i,\alpha_i)$ of $K$ such that $D_i$ is an oriented meridional disk on $\{1\}\times x_i$ and $\alpha_i$ is a line from a point of $\partial D_i$ to $*$.
Then, we see that the map sending $x_i \in X(D_K)$ to $[(D_i,\alpha_i)] \in X(K)$ is an isomorphism.
Since an $(n+2)$-ball $D^{n+2}$ can be identified with $\R^{n+2}_+ \cup \{\infty\}$, we have the following

\begin{proposition}\label{Proposition: pres. obtained by semi-sheets}
   For an proper embedded $n$-submanifold $K$ in $\R_+^{n+2}$ or $D^{n+2}$ and a diagram $D_K$ of $K$, the knot symmetric quandle $X(K)$ is isomorphic to $X(D_K)$.
\end{proposition}

In this paper, we identify $X(K)$ with $X(D_K)$ by the isomorphism given above.

\subsection{The knot symmetric quandle of braided surfaces}\label{subsection: knot quandle of braided surface}
First, we define Artin's automorphism for the free symmetric quandle.

We fix a Hurwitz arc system $(\alpha_1^1, \ldots, \alpha_n^1)$ on $D_1$ with the starting point set $Q_n$.
Let $x_i \in X(D_1, Q_n)$ be a homotopy class of a noose consisting of an oriented meridional disk of $q_i \in Q_n$ and the oriented arc obtained from $\alpha_i^1$ by restricting to the complement of the meridional disk of $q_i$.
Then $X(D_1, Q_n)$ is the free symmetric quandle generated by $x_1, \ldots, x_n$.

For a geometric $n$-braid $b \subset D_1\times I$, there is an isotopy $\{\varphi_t\}_{t\in [0,1]}$ of $D_1$ such that $\varphi_0 = \id$, $\varphi_t|_{\partial D_1} = \id$, and $\varphi_t(Q_n) = \mbox{pr}_1(b\cap (D_1\times \{t\}))$ for $0 \leq t \leq 1$, where $\mbox{pr}_1: D_1\times I \to D_1$ is the projection.
Then, a map sending $\beta = [b]\in B_n$ to an element $[\varphi_1]$ in the mapping class group of $(D_1, Q_n)$ is a well-defined (injective) homomorphism. See \cite{Kamada2002_book} for details.

Let $\Aut(X(D_1, Q_n))$ be the symmetric quandle automorphism group of $X(D_1, Q_n)$.
The braid group $B_n$ acts on $X(D_1, Q_n)$ by $\beta\cdot [(D,\alpha)] = [(\varphi_1(D), \varphi_1(\alpha))]$, which gives a group homomorphism $\mathrm{Artin}: B_n \to \Aut(X(D_1, Q_n))$.
For $\beta\in B_n$, $\mathrm{Artin}(\beta)$ is called \textit{Artinn's automorphism} (or \textit{braid automorphism}) of $X(D_1, Q_n)$.
Explicitly, for $i = 1,\ldots, n-1$, $\mbox{Artin}(\sigma_i)$ is written as
\[
   \mbox{Artin}(\sigma_i)(x_j) ~=~ \left\{
      \begin{array}{lll}
         x_{i+1}^{x_i^{-1}} & (j=i),\\
         x_{i} & (j=i+1),\\
         x_j & \mbox{otherwise}.
      \end{array}
   \right.
\]

Let $S$ be a braided surface of degree $n$ with $r$ branch points, and $(\beta_1^{-1} \sigma^{\varepsilon_1}_{k_1}\beta_1, \ldots, \beta_r^{-1} \sigma^{\varepsilon_r}_{k_r}\beta_r)$ be a braid system of $S$, where $\beta_i \in B_{2m}$, $\varepsilon_i \in \{\pm1\}$, and $1 \leq k_i \leq 2m-1$.

\begin{proposition}\label{Proposition: pres. of knot quandle of braided surface}
   $X(S)$ has the following presentation:
   \[
      \left\langle
         \begin{array}{c|}
            x_1, \ldots, x_{n}
         \end{array}
         \begin{array}{l}
            \mathrm{Artin}(\beta_j)(x_{k_j}) = \mathrm{Artin}(\beta_j)(x_{k_j + 1})~(j = 1, \ldots, r)
         \end{array}
      \right\rangle_{\mathrm{sq}}.
   \]
   % where each generator $x_i$ is the image of $x_i\in X(D_1,Q_n)$ by the homomorphism $X(D_1,Q_n) \to X(S)$ induced from the inclusion map $(D_1\times \{y_0\}, Q_n\times \{y_0\}) \hookrightarrow (D_1 \times D_2, S)$.
   % where each generator $x_i$ represents a noose at $(q_i, y_0) \in S \subset D_1\times D_2$.
\end{proposition}

This proposition can be proved by rephrasing the proof of \cite[Proposition~27.2]{Kamada2002_book} in the sense of symmetric quandles.
Thus we only sketch an outline of the proof.

\begin{proof}
   Let $S$ be a braided surface of degree $n$, $\Sigma(S) = \{ y_1, \ldots, y_r\}$ the branch locus of $\pi_S$, and $\mathcal{A} = (\alpha_1,\ldots, \alpha_n)$ a Hurwitz arc system on $D_2$ with the starting point set $\Sigma(S)$.
   We denote $(b_1, \ldots, b_n)$ a braid system of $S$ associated with $\mathcal{A}$, where $b_j = \beta_j^{-1} \sigma_{k_j}^{\varepsilon_j} \beta_j \in B_n$.

   Let $N(y_j)$ be a closed regular neighborhood of $y_j$ in $D_2$, and $y_0^{[j]} = \partial N(y_j) \cap \alpha_j$ be a based point of $N(y_j)$.
   We may assume that $S\cap (D_1\times \{y_0^j\}) = Q_n \times \{y_0^j\}$.
   Then $S_j^{(1)} = S \cap (D_1\times N(y_j))$ is a braided surface over $N(y_j)$ with a single branch point.
   Let $W_j^{(1)} = N(y_j)\times D_2$, and $x_i^{j} \in X(W_j^{(1)}, S_j^{(1)})$ denote the element obtained from $x_i\in X(D_1, Q_n)$ by the natural identification between $Q_n$ and $Q_n\times y_0^{[j]} \subset W_j$.
   Then, $(\sigma_{k_j}^{\varepsilon_J})$ is a braid system of $S_j^{(1)}$ and we can see that
   \[
      X(W_j^{(1)}, S_j^{(1)}) ~=~ \left\langle
      \begin{array}{c|}
         x_1^j, \ldots, x_{n}^j
      \end{array}
      \begin{array}{l}
         x_{k_j}^j = x_{k_j + 1}^j
      \end{array}
   \right\rangle_{\mathrm{sq}}.
   \]

   Next, let $N(\alpha_j)$ be a closed regular neighborhood of $\alpha_j$ in $D_2$ and $W_j^{(2)} = D_1\times N(\alpha_j)$.
   Then, $S_j^{(2)} = S \cap W_j^{(2)}$ is again a braided surface, and we have $x_i^j = \mbox{Artin}(\beta_j)(x_i)$ in $X(W_j^{(2)},S_j^{(2)})$.
   Thus, for each $j = 1, \ldots, r$, we have
   \[
      X(W_j^{(2)}, S_j^{(2)}) ~=~ \left\langle
      \begin{array}{c|}
         x_1, \ldots, x_{n}
      \end{array}
      \begin{array}{l}
         \mbox{Artin}(\beta_j)(x_{k_j}) = \mbox{Artin}(\beta_j)(x_{k_j+1})
      \end{array}
   \right\rangle_{\mathrm{sq}}.
   \]

   Finally, Let $W^{(3)} = D_1\times \cup_{j=1}^r N(\alpha_i)$.
   Then $S^{(3)} = \cup_{j=1}^r S_j^{(2)}$ is a braided surface over $\cup_{j=1}^r N(\alpha_i)$, and $X(W^{(3)}, S^{(3)})$ has the following presentation:
   \[
      X(W^{(3)}, S^{(3)}) ~=~ \left\langle
         \begin{array}{c|}
            x_1, \ldots, x_{n}
         \end{array}
         \begin{array}{l}
            \mathrm{Artin}(\beta_j)(x_{k_j}) = \mathrm{Artin}(\beta_j)(x_{k_j + 1})~(j = 1, \ldots, r)
         \end{array}
      \right\rangle_{\mathrm{sq}}.
   \]
   Since $(W^{(3)}, S^{(3)})$ is a deformation retract of $(D_1\times D_2, S)$, we complete the proof.
\end{proof}

\subsection{A proof of Theorem~\ref{MainTheoremA}}
Let $F$ be a surface-link in $\R^4$.
We may assume that $F$ is in a plat form, and we denote $S$ as an adequate braided surface of degree $2m$ such that $F = S \cup A_S$.
Let $L = \partial S$ be a classical link in a solid torus $D_1\times \partial D_2$.
By the projection $\pi: \R^4 = \R \times \R^3 \to \R^3$, we obtain a diagram $D_F$ of $F$.
We denote $D_S = D_F \cap \pi(S)$, $D_A = D_F \cap \pi(A_S)$, and $D_L = D_F \cap \pi(L)$.
Here, $D_F$ is a diagram of $L$ in an annulus $\pi(D_1\times \partial D_2)$.

Let $\sigma_{2i-1}$, $\tau_j$, and $\upsilon_j$ be adequate geometric $2m$-braids described in Figure~\ref{Figure: adequate braid gamma-j} ($i = 1, \ldots, m$, $j = 1, \ldots, m-1$), and $\sigma_{2i-1}^{-1}$, $\tau_j^{-1}$, and $\upsilon_j^{-1}$ be their inverses.
By deforming $S$ equivalently in $D_1\times D_2$, we may assume that $\beta_S$ is either the trivial $2m$-braid or a product of copies of $\sigma_{2i-1}^{\pm 1}$, $\tau_j^{\pm 1}$, and $\upsilon_j^{\pm 1}$.
Furthermore, by deforming $S$ equivalently in $D_1\times D_2$, we may add $\sigma_{2i-1} \sigma_{2i-1}^{-1}$ ($i = 1, \ldots, m-1$) to the end of $\beta_S$ so that the semi-sheet of $D_A$ containing $\pi(q_{2i}, y_0) \in \pi(D_1\times D_2)$ is different from the semi-sheet of $D_A$ containing $\pi(q_{2j}, y_0)$ for any $i \neq j$, where $q_{2i} \in Q_{2m}$.
We remark that the semi-sheet of $D_A$ containing $\pi(q_{2i}, y_0)$ also contians $\pi(q_{2i-1}, y_0)$.
As a result, $\beta_S$ is a product of $c$ adequate $2m$-braids $\gamma_1, \ldots, \gamma_c$.
Let $0 = t_0 < t_1 < \cdots < t_{c-1} < t_c = 1$ be the partition such that $\beta_S \cap (D_1 \times [t_{i-1},t_i]) = \gamma_i$, and let $A_i = A_S \cap ((D\times I)\times p([t_{i-1},t_i]))$ be the subsurface of $A_S$, where $p: I \to S^1 = I/\partial I$ is the quotient map (see Section~\ref{subsection: A plat form for surface-links}).
Let $D_{A, i} = \pi(A_i) \cap D_A$ denote a diagram of $A_i$, then $D_{A, i}$ consists of copies of subdiagram described in Figure~\ref{Figure: diagrams of A-j}.

\begin{figure}[h]
   \centering
   \includegraphics[width = \hsize]{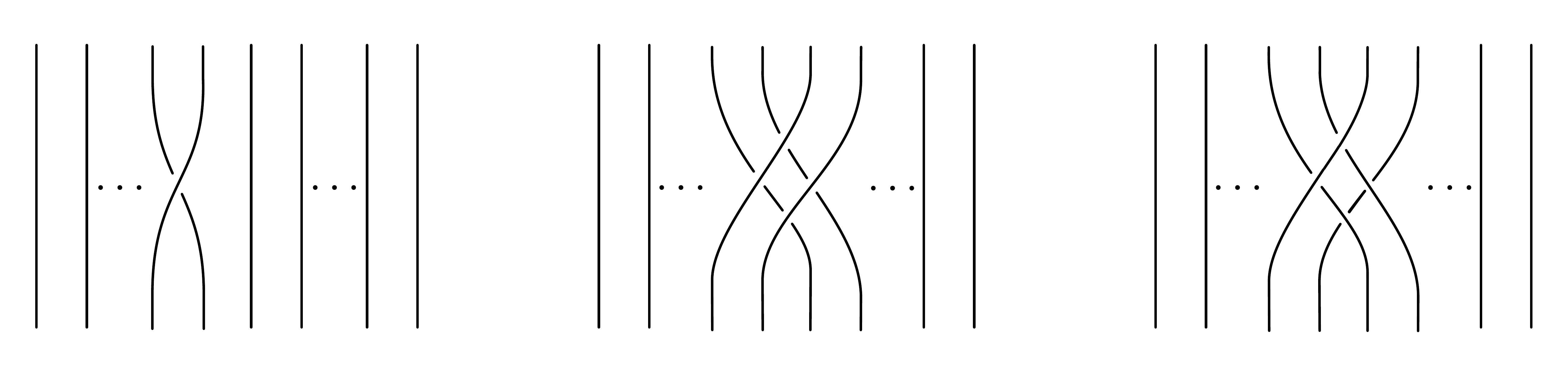}
   \caption{Geometric adequate $2m$-braids representing $\sigma_{2i-1}$ (left), $\tau_j = \sigma_{2j}\sigma_{2j-1}\sigma_{2j+1}\sigma_{2j}$ (center), and $\upsilon_j = \sigma_{2j}\sigma_{2j-1}\sigma_{2j+1}^{-1}\sigma_{2j}^{-1}$ (right) for $i = 1, \ldots, m$ and $j = 1, \ldots, m-1$.}
   \label{Figure: adequate braid gamma-j}
\end{figure}

\begin{lemma}\label{Lemma: Adjacency between double point curves in D-S and D-A}
   The singular set $\Delta_A$ of $\pi(A_S)$ is a disjoint union of simple arcs, and each simple arc in $\Delta_A$ intersects with $\pi(\partial A_S)$.
\end{lemma}

\begin{proof} % ok
   This lemma follows from the assumption that $D_A$ is a union of $D_{A,1}, \ldots, D_{A,c}$ which are unions of copies of subdiagrams as shown in Figure~\ref{Figure: diagrams of A-j}.
   See Figures~\ref{Figure: semi-sheets of s}, \ref{Figure: semi-sheets of k}, and \ref{Figure: semi-sheets of l} for descriptions of semi-sheets of them.
\end{proof}

\begin{figure}[h]
   \centering
   \includegraphics[width = 0.8\hsize]{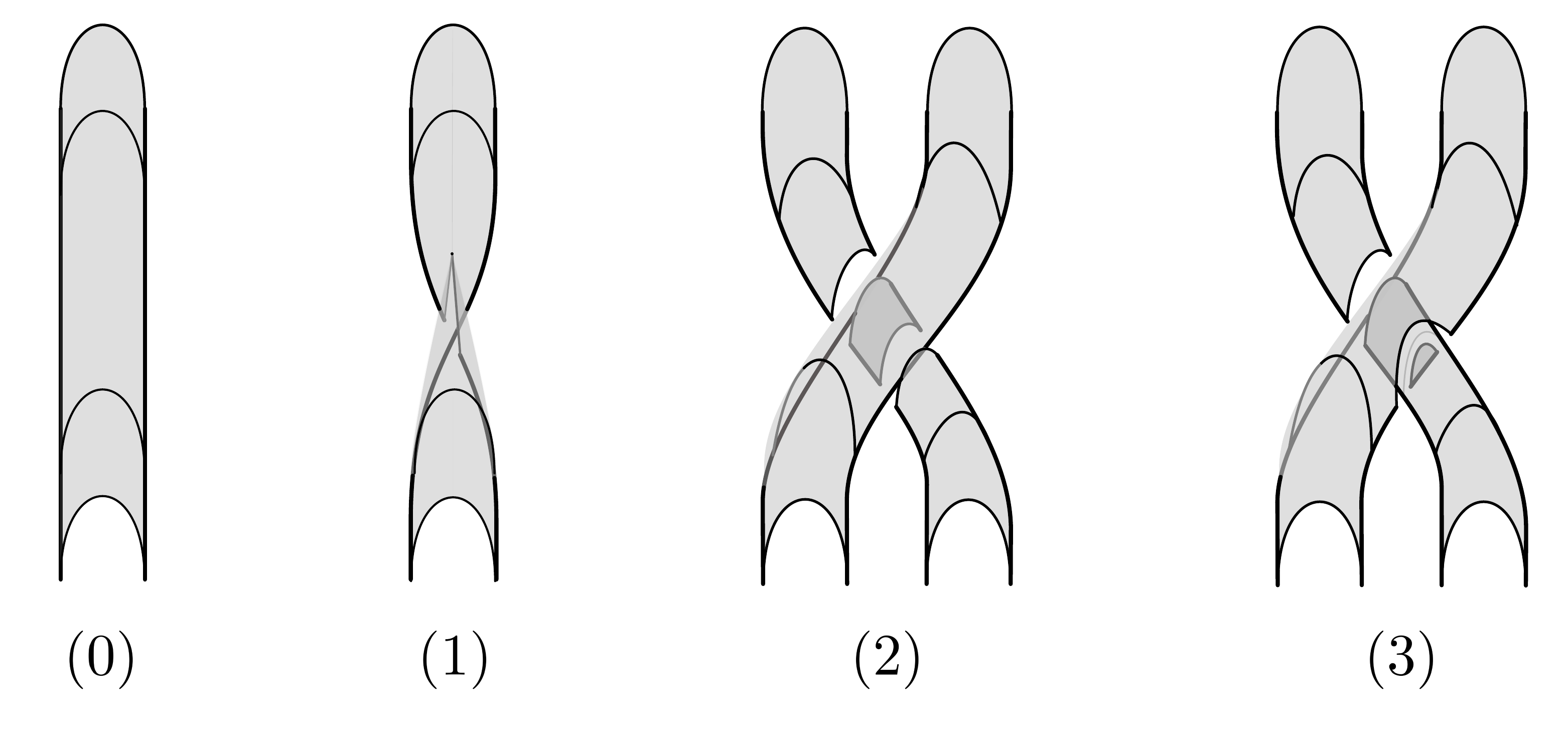}
   \caption{A diagram $D_{A,t}$ consists of one of (1), (2), (3), or their inverses and some copies of (0) ($t = 1,\ldots, c$).}
   \label{Figure: diagrams of A-j}
\end{figure}

Let $s$, $l$, $a$, and $f$ be the numbers of semi-sheets (or semi-arcs) of $D_S$, $D_L$, $D_A$, and $D_F$, respectively.
We denote semi-sheets (or semi-arcs) of $D_S$, $D_L$, $D_A$, and $D_F$ by $x_1, \ldots, x_s$, $y_1, \ldots, y_l$, $z_1, \ldots, z_a$, and $w_1, \ldots w_f$, respectively.
Since $S$ is a braided surface, a semi-sheet of $D_S$ containing $\pi(q_i, y_0)$ is different from a semi-sheet of $D_S$ containing $\pi(q_j, y_0)$ for any $i \neq j$.
This is true for $D_L$ and $D_A$.
Hence, by replacing subscripts of semi-sheets and semi-arcs, we assume that
\begin{itemize}
   \item for $h = 1, \ldots, 2m$, a semi-sheet $x_h$ contains $\pi(q_i, y_0)$,
   \item for $i = 1, \ldots, 2m$, a semi-arc $y_{i}$ contains $\pi(q_{i}, y_0)$, and
   \item for $j = 1, \ldots, m$, a semi-sheet $z_j$ contains $\pi(q_{2i}, y_0)$ and $\pi(q_{2i-1}, y_0)$.
\end{itemize}

We give semi-sheets of $D_A$ and $D_F$ orientations arbitrarily.
We give semi-sheets of $D_S$ the orientation induced from an orientation of $S$ which is compatible with an orientation of $D_2$.
We also give semi-arcs of $D_L$ the orientation induced from semi-sheets of $D_S$.

\begin{remark}\label{Remark: notation of orientations}
   The significance of giving such orientations is that for each wicket $w$ that forms $A_S$, the orientations of semi-sheets of $D_S$ containing endpoints of $w$ are incoherent across $w$.
\end{remark}

For each semi-arc $y_i$, there are unique semi-sheets $x_h$, $z_j$, and $w_k$ such that $y_i \subset x_h$, $y_i \subset z_j$, and $y_i \subset w_k$, respectively.
We denote $h(y_i) = h$, $j(y_i) = j$ and $k(y_i) = k$.
Similarly, for each semi-sheet $x_h$ (and $y_i$), there is a unique semi-sheet $w_k$ such that $x_h \subset w_k$ (and $y_i \subset w_k$), respectively.
We denote $k(x_h) = k$ and $k(y_i) = k$.

For each semi-sheet $x_h$, we define the sign $\varepsilon_h = 1$ if the orientation of $x_h$ is coherent with the orientation of $w_{k(x_h)}$ under the natural inclusion $x_h \subset w_{h(x_h)}$ and $\varepsilon_h = -1$ otherwise.
Similarly, for each semi-arc $y_i$ (and semi-sheet $z_j$), we define the signs $\zeta_i = 1$ (and $\eta_j = 1$) if the orientations of $x_{h(y_i)}$ (and $z_j$) are coherent with the orientations of $w_{k(y_i)}$ (and $w_{k(z_j)}$) under the natural inclusions $x_{h(y_i)} \subset w_{k(y_i)}$ (and $z_j \subset w_{k(z_j)}$) and $\zeta_i = -1$ (and $\eta_j = -1$) otherwise, respectively.

We denote $X_h = x_h^{\varepsilon_h}$, $Y_i = y_i^{\zeta_i}$, and $Z_j = z_j^{\eta_j}$, and put
\begin{align*}
   R_1^X ~=~ \{ Y_i = X_{h(y_i)} ~|~ i = 1, \ldots, l~\} \quad \mbox{and} \quad
   R_1^Z = \{ Y_i = Z_{j(y_i)} ~|~ i  = 1, \ldots, l~\}.
\end{align*}

\begin{lemma}\label{Lemma: DS DA R0}
   $X(F)$ has the following presentation:
   \begin{align}\label{align: X Y Z, DS R1}
   \left\langle
   \begin{array}{c|}
      x_1, \ldots, x_s,~
      y_1, \ldots, y_l,~ z_1, \ldots, z_a
   \end{array}
   \begin{array}{c}
      \mbox{A-rel}_{D_S}, \mbox{B-rel}_{D_S},
      R_1^X, R_1^Z
   \end{array}
   \right\rangle_{\mathrm{sq}}.
   \end{align}
   Here, each generators $x_h$, $y_i$, and $z_j$ in (\ref{align: X Y Z, DS R1}) are the images of $x_h \in X(S)$, $y_i \in X(L)$, and $z_j \in X(A)$ by the homomorphisms induced from the inclusion maps, respectively.
\end{lemma}

\begin{proof}
   % We start to construct a presentation of $X(A_S)$.
   From Proposition~\ref{Proposition: pres. obtained by semi-sheets}, $X(F)$ has the following presentation:
   \begin{align}\label{align: W, DF}
      X(F) ~=~ \left\langle
      \begin{array}{c}
         w_1, \ldots, w_f
      \end{array}
      \begin{array}{|c}
         \mbox{A-rel}_{D_F}, \mbox{B-rel}_{D_F}
      \end{array}
      \right\rangle_{\mathrm{sq}}.
   \end{align}

   We put $R_0^X = \{w_{k(x_h)} = X_h ~|~ h = 1, \ldots, s\}$ and $R_0^Z = \{ w_{k(z_j)} = Z_j ~|~ j = 1, \ldots, a\}$.
   We add new generators $x_1, \ldots, x_s$, $z_1, \ldots, z_a$, and $y_1, \ldots, y_l$ with $R_0^X$, $R_0^Z$, and $R_1^Z$ to (\ref{align: W, DF}):
   \begin{align}\label{align: X Y Z W, R0 DF}
      X(F) ~=~ \left\langle
   \begin{array}{c|}
      x_1, \ldots, x_s,~ y_1, \ldots, y_l,\\
      z_1, \ldots, z_a,~ w_1, \ldots w_f
   \end{array}
   \begin{array}{c}
      \mbox{A-rel}_{D_F}, \mbox{B-rel}_{D_F},
      R_0^X, R_0^Z, R_1^Z
   \end{array}
   \right\rangle_{\mathrm{sq}}.
   \end{align}
   % For convenience, we refer to $Y_i = Z_{k(i)} \in R_0$ as the \textit{relation for $y_i$}.

   Each double point stratum of $D_S$ is contained in one of $D_F$.
   Hence A-relations (and B-relations) for $D_S$ are consequences of an A-relation (and B-relation) for $D_F$ and relators in $R_0^X$, respectively.
   % Similarly, each crossing of $D_L$ is contained in some double point stratum of $D_S$.
   % So we can add $\mbox{A-rel}_{D_L}$ and $\mbox{B-rel}_{D_L}$ to (\ref{align: X Y Z W, R0 DF}) by Tietze moves (T1).
   Conversely, it follows from Lemma~\ref{Lemma: Adjacency between double point curves in D-S and D-A} that every double point stratum of $D_F$ contains one of $D_S$.
   This implies that A-relations (and B-relation) for $D_F$ is a consequence of A-relations (and B-relation) for $D_S$ and relations in $R_0^X$.
   Thus, we replace $\mbox{A-rel}_{D_F}$ and $\mbox{B-rel}_{D_F}$ in (\ref{align: X Y Z W, R0 DF}) with $\mbox{A-rel}_{D_S}$ and $\mbox{B-rel}_{D_S}$:
   \begin{align}\label{align: X Y Z W, DS R0}
      X(F) ~=~ \left\langle
   \begin{array}{c|}
      x_1, \ldots, x_s,~ y_1, \ldots, y_l,\\
      z_1, \ldots, z_a,~ w_1, \ldots w_f
   \end{array}
   \begin{array}{c}
      \mbox{A-rel}_{D_S}, \mbox{B-rel}_{D_S},
      R_0^X, R_0^Z, R_1^Z
   \end{array}
   \right\rangle_{\mathrm{sq}}.
   \end{align}

   Since $R_1^X$ consists of consequences of relators in (\ref{align: X Y Z W, DS R0}),
   we add $R_1^X$ to (\ref{align: X Y Z W, DS R0}):
   \begin{align}\label{align: X Y Z W, R0 R1 DS}
      X(F) ~=~ \left\langle
         \begin{array}{c|}
            x_1, \ldots, x_s,~ y_1, \ldots, y_l,\\
            z_1, \ldots, z_a,~ w_1, \ldots, w_f
         \end{array}
         \begin{array}{c}
            \mbox{A-rel}_{D_S}, \mbox{B-rel}_{D_S}.
            R_0^X, R_0^Z, R_1^X, R_1^Z
         \end{array}
         \right\rangle_{\mathrm{sq}}.
      \end{align}

   By Lemma~\ref{Lemma: Adjacency between double point curves in D-S and D-A}, each semi-sheet $w_k$ contains some semi-sheet $x_h$ of $D_S$.
   Hence there exists a relator $w_{k(x_h)} = X_h$ in $R_0^X$ such that $k(x_h) = k$.
   We fix such a relator for each $w_k$ and denote it as $r(w_k)$.
   Let $R_0^W = \{ r(w_k) ~|~ k = 1, \ldots, f\}$.

   Let $w_{k(x_h)} = X_h$ be a relator in $R_0^X\setminus R_0^W$.
   We will show that $w_{k(x_h)} = X_h$ is a consequence of $R_0^W$, $R_1^X$, and $R_1^Z$.
   By definition, there is a relator $w_k = X_{h_0}$ in $R_0^W$ such that $k(x_h) = k$.
   % Since the semi-sheet $w_k$ is connected, we consider a path $p$ on $w_k$ such that its endpoints are a point of the semi-sheet $x_h$ and a point of $x_{h_0}$.
   Let $p: [0,1] \to w_k$ be a path on the semi-sheet $w_k$ such that endpoints $p(0)$ and $p(1)$ are on semi-sheets $x_{h_0}$ and $x_h$, respectively.
   We may assume that the image of $p$ intersects with semi-arcs $y_{i_1}, \ldots, y_{i_{2q}}$ of $D_L$ transversely on $w_k$, see Figure~\ref{Figure: partition of semi-sheet w-k}.
   For $u = 1, \ldots, q$, we denote $h_{u} = h(y_{i_{2u}})$ and $j_u = j(y_{i_{2u}})$.
   Then, $Y_{i_{2u}} = X_{h_u}$ and $Y_{i_{2u-1}} = X_{h_{u-1}}$ are relators in $R_1^X$, and
   $Y_{i_{2u}} = Z_{j_u}$ and $Y_{i_{2u-1}} = Z_{j_u}$ are relators in $R_1^Z$.
   Hence $Y_{i_1} = Y_{i_{2q}}$ is a consequences of $R_1^X$ and $R_1^Z$.
   Since $w_{k} = X_{h_0}$, $Y_{i_1} = X_{h_0}$, and $Y_{i_{2q}} = X_{h}$ are relators in $R_0^W$, $R_1^X$, and $R_1^Z$, respectively, $w_{k(x_h)} = X_h$ is a consequence of them.

   \begin{figure}[h]
      \centering
      \includegraphics[width = \hsize]{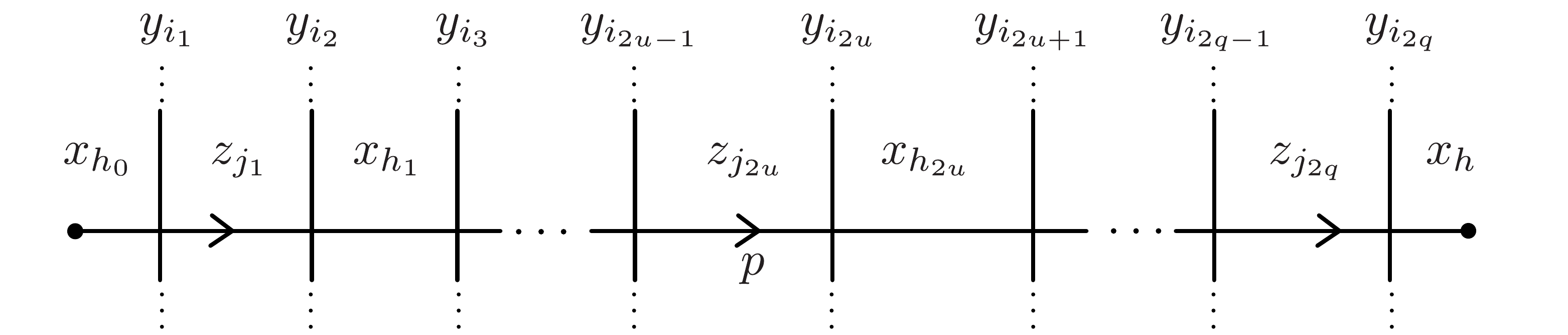}
      \caption{An illustration of intersections of the image of $p$ and semi-arcs of $D_L$ on the semi-sheet $w_k$.}
      \label{Figure: partition of semi-sheet w-k}
   \end{figure}

   Similarly, every relator in $R_0^Z$ is a consequence of $R_0^W$, $R_1^X$, and $R_1^Z$.
   Thus, we remove $R_0^X\setminus R_0^W$ and $R_0^Z$ from (\ref{align: X Y Z W, R0 R1 DS}):
   \begin{align}\label{align: X Y Z W, R0w R1 DS}
      X(F) ~=~ \left\langle
   \begin{array}{c|}
      x_1, \ldots, x_s,~ y_1, \ldots, y_l,\\
      z_1, \ldots, z_a,~ w_1, \ldots, w_f
   \end{array}
   \begin{array}{c}
      \mbox{A-rel}_{D_S}, \mbox{B-rel}_{D_S},
      R_0^W, R_1^X, R_1^Z
   \end{array}
   \right\rangle_{\mathrm{sq}}.
   \end{align}
   Then generators $w_1, \ldots, w_f$ do not occur in relators in (\ref{align: X Y Z W, R0w R1 DS}) except for $R_0^W$ so that we remove $w_1, \ldots, w_f$ and $R_0^W$ from (\ref{align: X Y Z W, R0w R1 DS}).
   Hence we obtain (\ref{align: X Y Z, DS R1}).
\end{proof}

Next, we define a subset $R_2$ of $R_1^Z$.
Let $Q = \{Y_{2j-1} = Z_{j}, Y_{2j} = Z_{j} ~|~ j = 1, \ldots, m \}$ be a subset of $R_1^Z$.
If $m = a$, we simply define $R_2 = R_1^Z = Q$, where $m$ is the half of the degree of $S$.
If $m < a$, we need to fix a relator $r_j: Y_{i} = Z_j$ in $R_1^Z$ for $j = m+1, \ldots, a$ as follows:

Let $t(y_i)$ (and $t(z_j)$) be the minimum numbers of second subscripts of diagrams $D_{A, t}$ which intersect with $y_i$ (and $z_j$) in $D_A$, respectively.
Clearly, $t(y_i) \geq t(z_{j(y_i)})$ holds for any $i$.
In addition, for $j = 1,\ldots, a$, there exists a semi-sheet $y_i$ such that $t(y_i) = t(z_j)$.
For $j = 1, \ldots, a$ and $t = 1, \ldots, c$, we put
\begin{align*}
   R_1^{t} ~&=~ \{ Y_i = Z_{j(y_i)} ~|~ t(i)=t ~(i = 1, \ldots, l)\},~\mbox{and}\\
   R_1^{t,j} ~&=~ \{ Y_i = Z_{j(y_i)} ~|~ t(i)=t,~ j(y_i)=j  ~(i = 1, \ldots, l)\}.
\end{align*}
Here, a relator $Y_i = Z_k$ belonging to $R_1^{t}$ ($t\geq 1$) indicates that the semi-arc $y_i$ intersects with $D_{A,t}$ and does not intersect with $D_{A, 1}, \ldots, D_{A,t-1}$.
By definition, $R_1^{t,j}$ is the empty set for any $t < t(z_j)$, and $R_1^{t(z_j), j}$ is a non-empty set.
For each $j = m+1, \ldots, a$, we fix a relator $r_j \in R_1^{t(z_j), j}$, and we define
\[
   R_2 = Q ~\cup~ \left\{r_{m+1}, r_{m+2}, \ldots, r_a \right\}.
\]

\begin{lemma}\label{Lemma: X Y Z, DS R1 R2}
   $X(F)$ has the following presentation:
   \begin{align}\label{align: X Y Z, DS R1 R2}
      \left\langle
         \begin{array}{c|}
            x_1, \ldots, x_s,~
            y_1, \ldots, y_l,~ z_1, \ldots, z_a
         \end{array}
         \begin{array}{l}
            \mbox{A-rel}_{D_S}, \mbox{B-rel}_{D_S}, R_1^X, R_2
         \end{array}
      \right\rangle_{\mathrm{sq}}.
   \end{align}
\end{lemma}

\begin{proof}
   We will show that relators in $R_1^Z \setminus R_2$ are consequences of relators in (\ref{align: X Y Z, DS R1 R2}).
   %  (\ref{align: X Y Z, DS R1 R2}) is obtained from (\ref{align: X Y Z, DS R1}) by Tietze moves (T2).
   Since A-relations (and B-relations) for semi-arcs of $D_L$ are consequences of relators in $\mbox{A-rel}_{D_S}$ (and $\mbox{B-rel}_{D_S}$) and $R_1^X$, respectively, we use relators in $\mbox{A-rel}_{D_L}$ and $\mbox{B-rel}_{D_L}$ instead of $\mbox{A-rel}_{D_S}$ and $\mbox{B-rel}_{D_S}$ in the following argument.

   $R_1^Z$ is divided into $c$ subsets $R_1^{1}, R_1^{2}, \ldots, R_1^{c}$.
   Thus, it suffices to show that relators in $R_1^{t}\setminus (R_2 \cap R_1^{t})$ are consequences of relators in (\ref{align: X Y Z, DS R1 R2}).
   We will show this by induction on $t$.
   We put $R_2^{0} = R_2$ and $R_2^{t} = R_2 \cap (\bigcup_{s=1}^{t} R_1^{s})$ for $t \geq 1$.
   Then, $R_2^{t-1}$ consists of consequences of relators in (\ref{align: X Y Z, DS R1 R2}) obtained from the assumption of induction at $t$.
   Since the arguments for the case when $t = 1$ and the case when $t \geq 2$ are parallel, we will discuss both cases simultaneously in the following.

   The diagram $D_{A,t}$ is a disjoint union of one of Figure~\ref{Figure: diagrams of A-j}-(1), -(2), -(3), or their inverses and some copies of Figure~\ref{Figure: diagrams of A-j}-(0).
   If a semi-arc $y_i$ intersects with a copy of Figure~\ref{Figure: diagrams of A-j}-(0) in $D_{A,t}$, then it holds that $t(y_i) < t$ and $t(z_{j(y_i)}) < t$.
   Thus $Y_i = Z_{j(y_i)}$ belongs to $R_2^{t-1}$.
   Hence we consider a relator $Y_i = Z_{j(y_i)}$ in $R_1^Z$ such that $y_i$ intersects with one of Figure~\ref{Figure: diagrams of A-j}-(1), -(2), -(3), or their inverses in $D_{A, t}$.

   (1) We first consider the case $\gamma_t = \sigma_{2k-1}$ ($k = 1, \ldots, m$), that is, $D_{A,t}$ contains a subdiagram as shown in the center of Figure~\ref{Figure: semi-sheets of s}.
   The right of Figure~\ref{Figure: semi-sheets of s} is shown semi-sheets of its subdiagram.
   It holds that $1 \leq j_1 \leq m$ when $t=1$.
   For $t \geq 2$, the semi-sheet $z_{j_1}$ intersects with $D_{A,t-1}$, i.e., $t(z_{j_1}) < t$.
   The semi-sheet $z_{j_2}$ might be satisfied $1\leq j_2 \leq m$.
   In such a case, $R_1^{t}\setminus (R_2 \cap R_1^{t})$ is the empty set.
   Thus we exclude such a case and assume that $j_2$ is more than $m$, which implies $t(z_{j_2}) = t$ and $t(y_{i_3}) = t(y_{i_4}) = t$.
   Hence we have
   \begin{alignat*}{2}
      R_1^{t, j_1} ~&=~ \left\{
         \begin{array}{ll}
            \{ Y_{i_1} = Z_{j_1}, Y_{i_2} = Z_{j_1} \} &(t=1),\\
            \emptyset &(t\geq 2),
         \end{array}
      \right.\\
      R_1^{t,j_2} ~&=~ \{ Y_{i_3} = Z_{j_2}, Y_{i_4} = Z_{j_2} \} \quad (t \geq 1),\\
      R_1^{t} ~&=~ \left\{
         \begin{array}{ll}
            Q ~\cup~  R_1^{t,j_2} &(t=1),\\
            R_1^{t,j_2} &(t\geq 2).
         \end{array}
      \right.\\
      R_2 \cap R_1^{t} ~&=~ \left\{
         \begin{array}{ll}
            Q ~\cup~ \left\{ r_{j_2} \right\} &(t=1),\\
            \left\{ r_{j_2} \right\} &(t\geq 2),
         \end{array}
      \right.
   \end{alignat*}
   where $r_{j_2} \in R_2$ is a fixed relator in $R_1^{t,j_2}$.
   Thus $R_1^{t}\setminus (R_2 \cap R_1^{t})$ consists of one relator, and such a relator, denoted by $\overline{r_{j_2}}$, comes from $R_1^{t,j_2}$.
   \begin{figure}[h]
      \centering
      \includegraphics[height = 35mm]{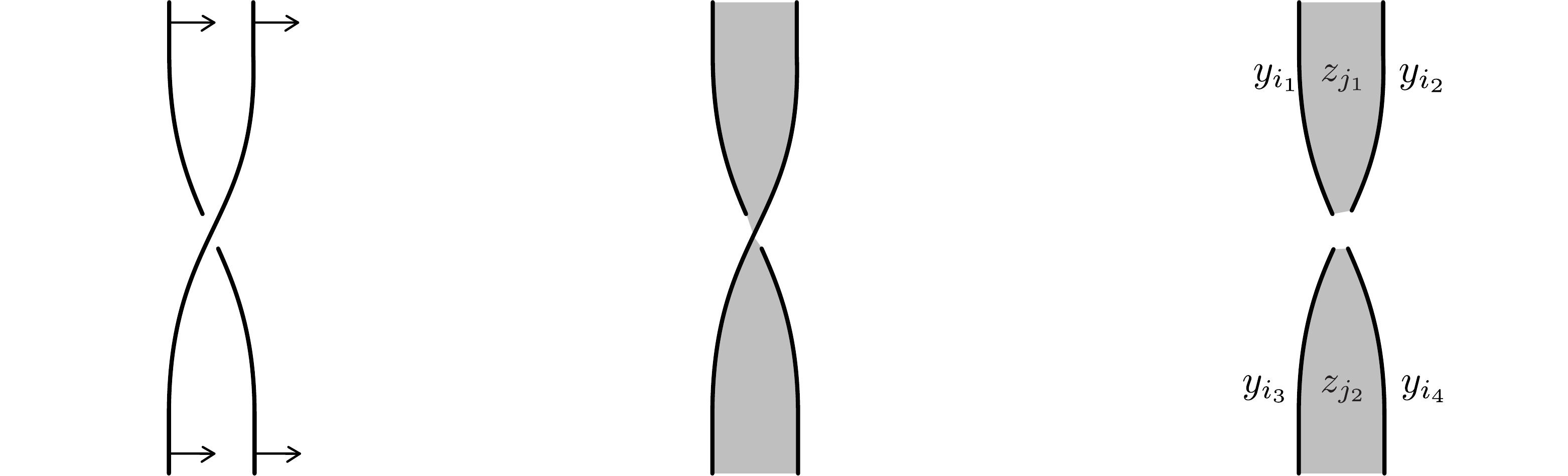}
      \caption{A diagram of a part of $\sigma_{2k-1}$ with normal orientations (left), a diagram of a part of $A_t$ (center), and semi-sheets of a part of $A_t$ (right).}
      \label{Figure: semi-sheets of s}
   \end{figure}

   Then, we can compute that $Y_{i_3} = Y_{i_4}$ is a consequence of relators in $\mbox{A-rel}_{D_L}$, $\mbox{B-rel}_{D_L}$, and $R_2^{t-1}$:
   \begin{equation}
      \begin{split}
         \label{align: a consequence of s}
         Y_{i_4} ~=~ y_{i_4}^{\zeta_{i_4}}
         ~\stackrel{B}{=}~ (y_{i_1}^{y_{i_2}})^{\zeta_{i_4}}
         ~\stackrel{(*)}{=}~ (y_{i_1})^{\zeta_{i_4}}
         ~\stackrel{R}{=}~ Z_{j_1}^{\zeta_{i_1}\zeta_{i_4}}
         ~\stackrel{R}{=}~ y_{i_2}^{\zeta_{i_2}\zeta_{i_1}\zeta_{i_4}}
         ~\stackrel{}{=}~ y_{i_2}^{-\zeta_{i_4}}
         ~\stackrel{A}{=}~ y_{i_3}^{-\zeta_{i_4}}
         ~\stackrel{}{=}~ y_{i_3}^{\zeta_{i_3}} ~=~ Y_{i_3}.
      \end{split}
   \end{equation}
   Here, two elements are connected by $\stackrel{A}{=}$, $\stackrel{B}{=}$, and $\stackrel{R}{=}$ if the equality is given from a relator in $\mbox{A-rel}_{D_L}$, $\mbox{B-rel}_{D_L}$, and $R_2^{t-1}$, respectively.
   The equality expressed by ($*$) is obtained from a result of $y_{i_2} = z_{j_1}^{\eta_{j_1}\zeta_{i_2}} = y_{i_1}^{\zeta_{i_1}\zeta_{i_2}} = y_{i_1}^{-1}$.
   Hence, $\overline{r_{j_2}}$ is a consequence of $r_{j_2}$ and $Y_{i_3} = Y_{i_4}$, that is, $R_1^{t}\setminus(R_2 \cap R_1^{t})$ consists of a consequence of relators in (\ref{align: X Y Z, DS R1 R2}).
   Similarly, we can prove it in the case of $\gamma_t = \sigma_{2k-1}^{-1}$.

   (2) Next, we consider the case $\gamma_t = \tau_k$ ($k = 1, \ldots, m-1$), that is, $D_{A,t}$ contains a subdiagram as in the center of Figure~\ref{Figure: semi-sheets of k}.
   The right of Figure~\ref{Figure: semi-sheets of k} is shown semi-sheets of the subdiagram.

   It holds that $1\leq j_1 \leq m$ and $1\leq j_4 \leq m$ when $t = 1$.
   For $t \geq 2$, $t(y_{i_1}), t(y_{i_2}), t(y_{i_7})$, and $t(t_{i_8})$ are less than $t$.
   For $t \geq 1$, $j_2$, $j_5$, and $j_6$ are more than $m$.
   For a similar reason as in (1), we may assume that $j_3$ is also more than $m$ for $t \geq 1$ and that $j_4$ is also more than $m$ for $t \geq 2$.
   Then, we have
   \begin{alignat}{2}
      \begin{split}\label{align: subset R-ikt case k}
      R_1^{t,j_1} ~&=~ \left\{
         \begin{array}{ll}
            \{ Y_{i_1} = Z_{j_1}, Y_{i_2} = Z_{j_1} \} &(t=1),\\
            \emptyset &(t\geq 2),
         \end{array}
      \right.\\
      R_1^{t,j_2} ~&=~\{ Y_{i_3} = Z_{j_2}, Y_{i_4} = Z_{j_2} \}, \quad\,
      R_1^{t,j_3} ~=~ \{ Y_{i_5} = Z_{j_3}, Y_{i_6} = Z_{j_3} \}, \quad (t\geq 1)\\
      R_1^{t,j_4} ~&=~ \left\{
         \begin{array}{ll}
            \{ Y_{i_{11}} = Z_{j_4}, Y_{i_{12}} = Z_{j_4}, Y_{i_7} = Z_{j_4}, Y_{i_8} = Z_{j_4} \} &(t=1),\\
            \{ Y_{i_{11}} = Z_{j_4}, Y_{i_{12}} = Z_{j_4} \} &(t\geq 2),
         \end{array}
      \right.\\
      R_1^{t,j_5} ~&=~ \{ Y_{i_9} = Z_{j_5} \}, \quad
      R_1^{t,j_6} ~=~ \{ Y_{i_{10}} = Z_{j_6} \} \quad (t\geq 1),\\
      R_1^{t} ~&=~ \left\{
         \begin{array}{ll}
            Q \cup R_1^{t,j_1}\cup \cdots \cup R_1^{t,j_6} &(t=1),\\
            R_1^{t,j_1}\cup \cdots \cup R_1^{t,j_6} &(t\geq 2),
         \end{array}
      \right.\\
      R_2 \cap R_1^{t} ~&=~ \left\{
         \begin{array}{ll}
            Q \cup \{ r_{j_2}, r_{j_3}, Y_{i_9} = Z_{j_5}, Y_{i_{10}} = Z_{j_6} \} &(t=1),\\
            \{ r_{j_2}, r_{j_3}, Y_{i_9} = Z_{j_5}, Y_{i_{10}} = Z_{j_6} \} &(t\geq 2).
         \end{array}
      \right.
   \end{split}
   \end{alignat}
   Let $\overline{r_{j_2}}$ and $\overline{r_{j_3}}$ be relators in $R_1^{t,j_2}$ and $R_1^{t,j_3}$ which are not $r_{j_2}$ and $r_{j_3}$, respectively.
   Then, $R_1^{t}\setminus (R_2 \cap R_1^{t})$ consists of four relators, which are $Y_{i_{11}} = Z_{j_4}$, $Y_{i_{12}} = Z_{j_4}$, $\overline{r_{j_2}}$, and $\overline{r_{j_3}}$.

   \begin{figure}[h]
      \centering
      \includegraphics[height = 35mm]{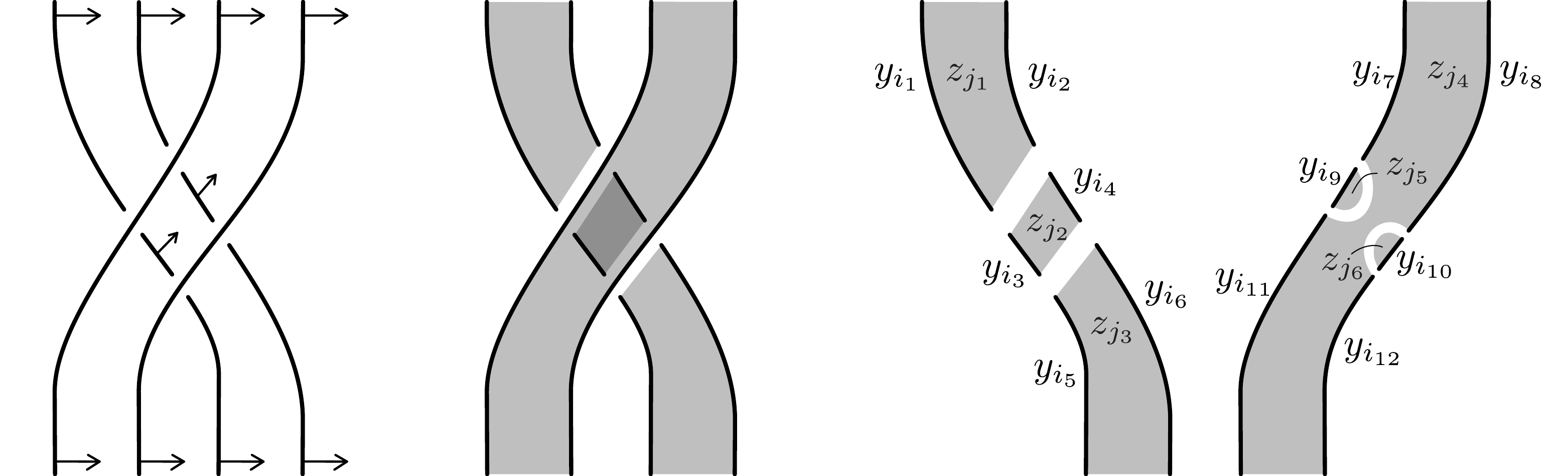}
      \caption{A diagram of a part of $\tau_{k}$ with normal orientations (left), a diagram of a part of $A_1$ (center), and semi-sheets of a part of $A_t$ (right).}
      \label{Figure: semi-sheets of k}
   \end{figure}

   By similar computations as shown in (\ref{align: a consequence of s}), both $Y_{i_3} = Y_{i_4}$ and $Y_{i_5} = Y_{i_6}$ are consequences of relators in (\ref{align: X Y Z, DS R1 R2}).
   Hence, $\overline{r_{j_2}}$ and $\overline{r_{j_3}}$ are also consequences of them.
   By the following computations, $Y_{i_{11}} = Z_{j_4}$ and $Y_{i_{12}} = Z_{j_4}$ are also consequences of them:
   \[
      Y_{i_{11}} ~\stackrel{A}{=}~ Y_{i_9} ~\stackrel{A}{=}~ Y_{i_7} ~\stackrel{R}{=}~ Z_{j_{4}}, \quad
      Y_{i_{12}} ~\stackrel{A}{=}~ Y_{i_{10}} ~\stackrel{A}{=}~ Y_{i_8} ~\stackrel{R}{=}~ Z_{j_{4}}.
   \]
   Therefore, we see that all elements of $R_1^{t}\setminus (R_2 \cap R_1^{t})$ are consequences of (\ref{align: X Y Z, DS R1 R2}).
   This computation can be applied to the case $\gamma_t = \tau_k^{-1}$.

   (3) Finally, we consider the case $\gamma_t = \upsilon_k$ ($k = 1, \ldots, m$), that is, $D_{A,t}$ contains a subdiagram as in the center of Figure~\ref{Figure: semi-sheets of l}.

   \begin{figure}[h]
      \centering
      \includegraphics[height = 35mm]{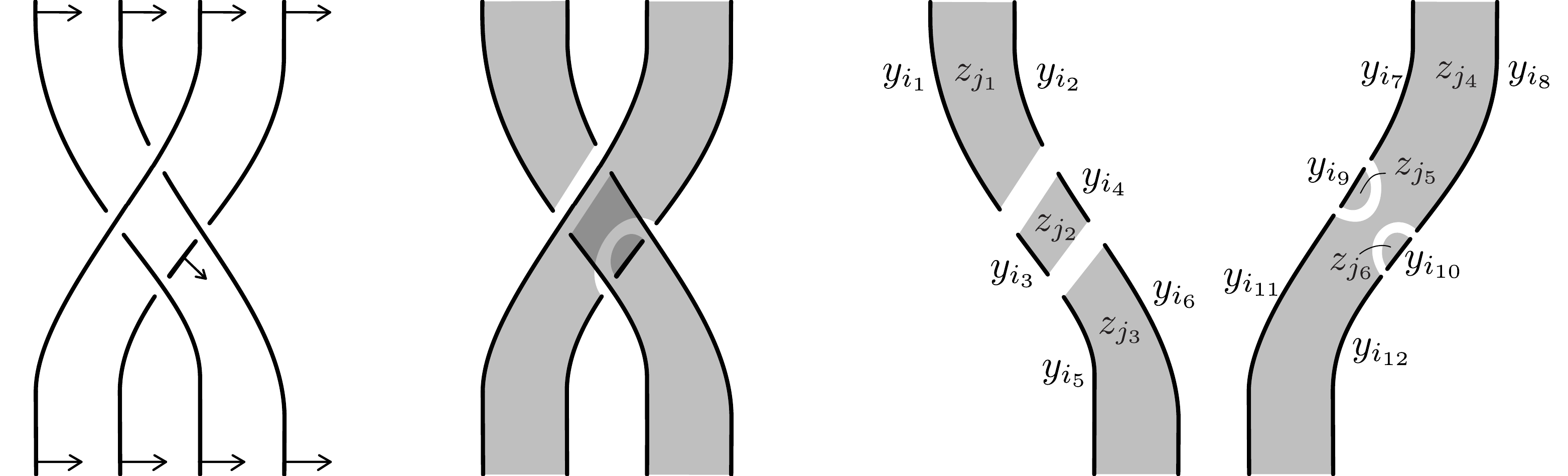}
      \caption{A diagram of a part of $\upsilon_{k}$ with normal orientations (left), a diagram of a part of $A_t$ (center), and semi-sheets of a part of $A_t$ (right).}
      \label{Figure: semi-sheets of l}
   \end{figure}

   Since semi-sheets are same for the case (2), $R_1^{t}$ and $R_2 \cap R_1^{t}$ are the same as (\ref{align: subset R-ikt case k}).
   Hence we can see that $R_1^{t}\setminus R_1$ consists of four elements and that these relators are consequences of relators in (\ref{align: X Y Z, DS R1 R2}).
   This computation can be applied when $\gamma_t = \upsilon_{k}^{-1}$.

   Therefore, we remove relators in $R_1^Z\setminus R_2$ from (\ref{align: X Y Z, DS R1}), and obtain (\ref{align: X Y Z, DS R1 R2}).
\end{proof}

\begin{proof}[Proof of Theorem~\ref{MainTheoremA}]
   By Tietze moves (T4), we remove generators $z_{m+1}, \ldots, z_a$ and relators $r_{m+1}, \ldots, r_a \in R_2$ from (\ref{align: X Y Z, DS R1 R2}):
   \begin{align*}
         X(F) ~=~ \left\langle
      \begin{array}{c}
         x_1, \ldots, x_s,~
         y_1, \ldots, y_l,~ z_1, \ldots, z_{m}
      \end{array}
      \begin{array}{|c}
         \mbox{A-rel}_{D_S}, \mbox{B-rel}_{D_S}, R_1^X, Q
      \end{array}
      \right\rangle_{\mathrm{sq}},
   \end{align*}
   where $R_1^X = \{ X_{h(y_i)} = Y_i ~|~ 1 \leq i \leq l\}$ and $Q = \{ Y_{2k} = Z_k, Y_{2k-1} = Z_k ~|~ 1 \leq k \leq m \}$.

   We put $R_3 = \{ Y_{2k} = Z_k ~|~ 1 \leq k \leq m \}$ and $R_4 = \{ X_{2k-1} = X_{2k} ~|~ 1 \leq k \leq m \}$.
   Since $h(y_{2k}) = h(y_{2k-1}) = k$ for $k = 1, \ldots, 2m$, relators in $R_3$ and $R_4$ are consequences of $R_1^X$ and $Q$.
   Conversely, relators in $Q$ are consequences of $R_1^X$, $R_3$, and $R_4$.
   Thus, we replace $Q$ with $R_3$ and $R_4$.
   Furthermore, we remove generators $z_1, \ldots, z_{m}$ and $y_1, \ldots, y_{l}$ with $R_3$ and $R_1^X$:
   \begin{align*}
      X(F) ~=~ \left\langle
         \begin{array}{c}
            x_1, \ldots, x_s
         \end{array}
         \begin{array}{|c}
            \mbox{A-rel}_{D_S}, \mbox{B-rel}_{D_S}, R_4
         \end{array}
      \right\rangle_{\mathrm{sq}}.
   \end{align*}

   By Remark~\ref{Remark: notation of orientations}, $\varepsilon_{2k} = -\varepsilon_{2k-1}$ holds for $k = 1,\ldots, m$.
   Hence $R_4$ can be rewritten as
   \begin{align}\label{align: X, DS R4}
      X(F) ~=~ \left\langle
         \begin{array}{c}
            x_1, \ldots, x_s
         \end{array}
         \begin{array}{|c}
            \mbox{A-rel}_{D_S}, \mbox{B-rel}_{D_S}, \\
            x_{2k} = x_{2k-1}^{-1} ~(k  = 1, \ldots, m)
         \end{array}
      \right\rangle_{\mathrm{sq}}.
   \end{align}
   Thus we obtain Theorem~\ref{MainTheoremA} from (\ref{align: X, DS R4}) by applying Propositions~\ref{Proposition: pres. obtained by semi-sheets} and \ref{Proposition: pres. of knot quandle of braided surface}.
\end{proof}

\begin{remark}\label{Remark: pres of knot group}
   Since the knot group $G(F) = \pi_1(\R^4\setminus N(F))$ of $F$ is isomorphic to the associated group of the knot symmetric quandle of $F$ (\cite{Kamada2014}), Theorem~\ref{MainTheoremA} induces a group presentation of $G(F)$:
   \[
      G(F) ~\cong~ \left\langle x_1, \ldots, x_{2m}~
      \begin{array}{|c}
         \mathrm{Artin}(\beta_j)(x_{k_i}) = \mathrm{Artin}(\beta_j)(x_{k_i+1}) ~(i=1,\ldots,r)\\
         x_{2k-1}=x_{2k}^{-1} ~ (k = 1,\ldots,m)
      \end{array}
    \right\rangle_{\mathrm{grp}},
   \]
   where $\mathrm{Artin}: B_m \to \Aut(F_{2m})$ is Artin's braid automorphism (cf. \cite{Kamada2002_book}).
\end{remark}

\section{Proofs of Theorems~\ref{Theorem: inequality for plat index} and \ref{MainTheoremB}}\label{Section: Proof of main theorem B}

In this section, we give proofs of Theorems~\ref{Theorem: inequality for plat index} and \ref{MainTheoremB}.
The key concept behind the following proof of Theorem~\ref{Theorem: inequality for plat index} is based on \cite[Lemma~2.7]{Sato-Tanaka2022}.

% \begin{proposition}\label{Proposition ineq between coloring and plat index}
%    For any surface-link $F$ and any finite symmetric quandle $X = (X,\rho)$, the following inequality holds:
%    \[
%       \log_{|X|}\mathrm{col}_{X}(F) \leq \mathrm{Plat}(F).
%    \]
% \end{proposition}

\begin{proof}[Proof of Theorem~\ref{Theorem: inequality for plat index}]
   Let $F$ be a surface-link and $X$ be a finite symmetric quandle.
   Since $\mathrm{col}_{X}(F)$ is defined as the cardinal number of (symmetric quandle) homomorphisms from $X(F)$ to $X$, it suffices to show the inequality
   \[
      \# \mathrm{Hom}(X(F), X) \leq (\# X)^{\mbox{Plat}(F)}.
   \]
   Let $S$ be an adequate braided surface of degree $2m$ such that $\widetilde{S} \cong F$ and $m = \mathrm{Plat}(F)$.
   By Theorem~\ref{MainTheoremA}, the knot symmetric quandle $X(F)$ has a presentation with $2m$ generators $x_1, \ldots, x_{2m}$.
   We can remove a generator $x_{2j-1}$ with a relator $x_{2j} = x_{2j-1}^{-1}$ from the presentation ($j = 1, \ldots, m$).
   Thus the knot symmetric quandle $X(F)$ is generated by $m$ elements.
   Since a homomorphism is determined by the images of generators of the domain quandle, $\# \mathrm{Hom}(X(F), X)$ is less than or equal to $(\# X)^m$, i.e., we have the inequality.
\end{proof}

% \begin{theorem}
%    For any $m \geq 2$ and $g\geq 0$, there exists infinitely many distinct surface-knots $F$ of genus $g$ such that $\mathrm{Plat}(F) = \mathrm{g.Plat}(F) = m$.
% \end{theorem}

\begin{proof}[Proof of Theorem~\ref{MainTheoremB}]
   First, for an integer $m \geq 2$ and a prime number $p \geq 3$, we construct a 2-knot $F(m,p)$ as follows:
   Let $b(m,p)$ be a $(2m-2)$-tuple of $2m$-braids defined as
   \[
      b(m,p) ~=~ (\beta^{-1}\sigma_{1}\beta,~ \beta^{-1}\sigma_{1}^{-1}\beta,~ \beta^{-1}\sigma_{3}\beta,~ \beta^{-1}\sigma_{3}^{-1}\beta,~ \ldots,~ \beta^{-1}\sigma_{2m-3}\beta,~ \beta^{-1}\sigma_{2m-3}^{-1}\beta),
   \]
   where $\beta = (\sigma_2\sigma_4\cdots \sigma_{2m-2})^p \in B_{2m}$.
   Then, $b(m,p)$ is a braid system for a $2$-dimensional braid $S(m,p)$ of degree $2m$.
   Let $F(m,p)$ be the plat closure of $S(m,p)$.
   By definition, we have the inequality $\mathrm{Plat}(F(m,p)) \leq \mathrm{g.Plat}(F(m,p)) \leq m$.
   Furthermore, $F(m,p)$ is a 2-knot because $F(m,p)$ is connected and $\chi(F(m,p)) = m - (2m-2) + m = 2$.

   By Theorem~\ref{MainTheoremA}, the knot symmetric quandle $X(F)$ has the following presentation:
   \begin{align}\label{align presentation}
      X(F) ~=~ \left\langle x_1, \ldots, x_{2m}~
         \begin{array}{|cl}
            \mathrm{Artin}(\beta)(x_{2i-1}) = \mathrm{Artin}(\beta)(x_{2i}) ~ &(i = 1, \ldots, m-1)\\
            x_{2j-1}=\rho(x_{2j}) &(j = 1, \ldots, m)
         \end{array}
       \right\rangle_{\mathrm{sq}}.
   \end{align}

   Let $q \geq 3$ be a prime number and $(R_q, \id)$ be the symmetric dihedral quandle of order $q$.
   Let $f: A = \{x_1, \ldots, x_{2m}\} \to R_q$ be a map and $f_\sharp: \mbox{FSQ}(A) \to R_q$ be a homomorphism induced from $f$.
   We write $y_i = f_\sharp(x_i)$.
   Since $\beta = (\sigma_2\sigma_4\cdots \sigma_{2m-2})^p$, we have
   \[
      f_\sharp(\mathrm{Artin}(\beta)(x_k)) ~=~ \begin{cases}
         y_1 &(k=1),\\
         py_{2k+1} - (p-1)y_{2k} &(k = 2, 4, \ldots, 2m-2),\\
         (p+1)y_{2k-2} - py_{2k-1} &(k = 3, 5, \ldots, 2m-1),\\
         y_{2m} &(k=2m).
      \end{cases}
   \]
   By Lemma~\ref{Lemma: Universal property for symmetric quandle presentation}, $f$ extends to a homomorphism $f:X(F(m,p)) \to R_q$ if and only if the following equations hold:
   \begin{align*}
      \begin{array}{rll}
         y_1 &= ~~ py_3 - (p-1)y_2, \quad &\\
         (p+1)y_{2i-1} - py_{2i-2} &= ~~ py_{2i+1} - (p-1)y_{2i} \quad &(i = 2, \ldots, m-1),\\
         y_{2j-1} &= ~~ y_{2j} \quad &(j = 1,\ldots, m).
      \end{array}
   \end{align*}

   If $p \neq q$, the above equations lead to $y_1 = y_2 = \cdots = y_{2m}$.
   Hence, every symmetric quandle homomorphism is a constant function, i.e., $\mathrm{col}_{(R_q,\id)}(F(m,p)) = q$.
   Since $qy = 0$ holds for any $y \in R_q$, if $p = q$, then
   the above equations can be reduced to $y_{2j-1} = y_{2j}$ $(j = 1,\ldots, m)$, i.e., $\mathrm{col}_{(R_q,\id)}(F(m,p)) = q^m$.

   By Theorem~\ref{Theorem: inequality for plat index}, we have $\mathrm{Plat}(F(m,p)) = \mathrm{g.Plat}(F(m,p)) \geq m$.
   Thus, the plat index and genuine plat index of $F(m,p)$ are equal to $m$.
   In addition, for two distinct prime numbers $p, q \geq 3$, the $(R_p, \id)$-coloring number detects that 2-knots $F(m,p)$ and $F(m,q)$ are inequivalent.
   Hence we obtain infinitely many 2-knots with $\mathrm{Plat}(F(m,p)) = \mathrm{g.Plat}(F(m,p)) = m$.

   Finally, for an integer $g\geq 1$, we construct a surface-knot $F(m,p,g)$ of genus $g$ such that $\mathrm{Plat}(F(m, p, g)) = \mathrm{g.Plat}(F(m, p, g)) = m$.

   Let $b(p,m,g)$ be a $(2m-2+2g)$-tuple of $2m$-braids defined as
   % \[
   %    b(p,m,g) ~=~ (\beta^{-1}\sigma_{1}\beta,\,\ldots,\, \beta^{-1}\sigma_{1}\beta,\, \beta^{-1}\sigma_{1}^{-1}\beta,\,\ldots,\, \beta^{-1}\sigma_{1}^{-1}\beta,\, \beta^{-1}\sigma_{3}\beta,\, \beta^{-1}\sigma_{3}^{-1}\beta,\, \ldots,\, \beta^{-1}\sigma_{2m-3}\beta,\, \beta^{-1}\sigma_{2m-3}^{-1}\beta),
   % \]
   \[
      b(p,m,g) ~=~ b(m,p) \oplus (\beta^{-1} \sigma_{1} \beta, \ldots, \beta^{-1} \sigma_{1} \beta) \oplus (\beta^{-1} \sigma_{1}^{-1} \beta, \ldots, \beta^{-1} \sigma_{1}^{-1} \beta),
   \]
   where $\beta = (\sigma_2\sigma_4\cdots \sigma_{2m-2})^p \in B_{2m}$, and the second and third terms are $g$-tuples of $\beta^{-1} \sigma_{1} \beta$'s and $\beta^{-1} \sigma_{1}^{-1} \beta$'s, respectively.
   Then, $b(p,m,g)$ is a braid system of a $2$-dimensional braid $S(m,p,g)$ of degree $2m$.
   Let $F(m,p,g)$ be the plat closure of $S(m,p,g)$.
   Since $F(m,p,g)$ is connected, orientable and $\chi(F(m,p,g)) = m - (2m-2+2g) + m = 2-2g$, $F(m,p,g)$ is an orientable surface-link of genus $g$.

   By Theorem~\ref{MainTheoremA}, the knot symmetric quandle $X(F(m,p,g))$ is isomorphic to $X(F(m,p))$.
   Hence, the above argument implies that the plat index and the genuine plat index of $F(m,p,g)$ are equal to $m$, and $F(m,p,g)$ and $F(m,q,g)$ are inequivalent if $p \neq q$.
\end{proof}

\section*{Acknowledgment}
The author would like to thank Seiichi Kamada and Yuta Taniguchi for their helpful advice on this research.
This work was supported by JSPS KAKENHI Grant Number 22J20494.

% Since the symmetric quandle coloring number of the connected sum of a surface-link $F$ and a trivial orientable surface-knot is the same as that of $F$, for any $m \geq 2$, the odd prime number $p \geq 3$, and $g\geq 1$, $F(m,p) \# \Sigma_g$ is a surface-knot of genus $g$ whose

% Finally, we show that the connected sum of the 2-knot $F(m,p)$ and a trivial orientable surface-knot $\Sigma_g$ of genus $g$ satisfies $\mathrm{Plat}(F(m,p)\# \Sigma_g) = \mathrm{g.Plat}(F(m,p)\# \Sigma_g)$ for any $g \geq 1$.

\bibliographystyle{plain}
\bibliography{reference.bib}
\end{document}